\documentclass[preprint,12pt]{elsarticle}

\usepackage{amssymb}
\usepackage{amsmath}
\usepackage{amsthm}

\usepackage{paralist}
\usepackage{tikz-cd}
\usepackage{hyperref}
\usepackage{tikz}
\usetikzlibrary{positioning,arrows.meta,calc,shapes.geometric,fit,backgrounds}

\theoremstyle{definition}
\newtheorem{theorem}{Theorem}[section]
\newtheorem{proposition}[theorem]{Proposition}
\newtheorem{corollary}[theorem]{Corollary}
\newtheorem{example}[theorem]{Example}
\newtheorem{definition}[theorem]{Definition}
\newtheorem{lemma}[theorem]{Lemma}
\newtheorem{definition-proposition}[theorem]{Definition-Proposition}
\newtheorem{remark}[theorem]{Remark}

\journal{Journal of Abstract Nonsense}

\newcommand{\rk}{\mathrm{rk}\,}
\newcommand{\ran}{\mathrm{ran}\,}

\newcommand{\llangle}{\left\langle\!\left\langle}
\newcommand{\rrangle}{\right\rangle\!\right\rangle}
\newcommand{\bbot}{{\bot\!\!\!\bot}}

\newcommand{\set}[1]{\left\lbrace #1\right\rbrace}

\newcommand{\norm}[1]{\left\Vert #1\right\Vert}

\renewcommand{\exp}[1]{{\mathrm{e}^{#1}}}

\newcommand{\simto}{\xrightarrow{\raisebox{-0.7ex}[0ex][0ex]{$\sim$}}} 

\begin{document}

\numberwithin{equation}{section}

\begin{frontmatter}



\title{Interconnection of port-Hamiltonian systems is not dynamical}


\author[1]{Jonas Kirchhoff} 

\affiliation[1]{organization={Institut für Mathematik, Martin-Luther-Universität Halle-Wittenberg},
            addressline={Theodor-Lieser-Stra\ss e 5}, 
            city={Halle (Saale)},
            postcode={06120}, 
            state={Sachsen-Anhalt},
            country={Germany}}

\begin{abstract}
The interconnection of port-Hamiltonian systems is usually understood as structure-preserving due to the geometric fact that the composition of the underlying Dirac structures produces again a Dirac structure. Pointwise, this is doubtless true, but when comparing dynamical systems, their trajectories are usually compared, and trajectories usually obey some regularity assumptions. It is shown that both classical and weak solutions do not necessarily exhibit a one-to-one correspondence between solutions of the interconnected port-Hamiltonian system, and the dynamical system realising the interconnection. In the mostly studied cases of constant Dirac structures or port-Hamiltonian ODE systems, both systems are behaviourally equivalent; for non-constant Dirac structures, a sufficient condition is given.
\end{abstract}


\end{frontmatter}




\section{Introduction}

Port-Hamiltonian systems are widely used for modelling and control of all sorts of things. One key feature of port-Hamiltonian systems is the structure-preserving interconnection: Given a number of port-Hamiltonian systems, one can impose a class of feedback laws on (parts of) the external ports so that the closed loop is again a port-Hamiltonian system. This property makes port-Hamiltonian systems attractive for automatic modelling of complex systems, where the total system is comprised of interacting simpler partial systems, each of which and their interactions can be modelled in a port-Hamiltonian way. The paradigmatic examples of such systems are networks of components, e.g. electrical circuits, robot arms, or gas-networks.

\begin{figure}
\centering
\begin{tikzpicture}[
  >={Stealth[length=2.4mm]},
  block/.style={draw, thick, ellipse, align=center, inner sep=2pt,
                minimum height=14mm, minimum width=20mm},
  dirac/.style={draw, thick, ellipse, align=center, inner sep=2pt, 
  				minimum height=14mm, minimum width=20mm},
  port/.style={font=\footnotesize},
  ann/.style={font=\scriptsize, align=center},
  d/.style={4mm}                      
]

\node[dirac] (D) {Dirac\\structure\\$\mathcal{D}$};

\node[block, left=10mm of D]  (S) {energy reservoir};
\node[block, right=10mm of D] (R) {resistive port};
\node[below=7.5mm of D] (P) {};

\draw[->] ($(D.west)+(0, 2mm)$) -- node[port,above]{$f_S$} ($(S.east)+(0, 2mm)$);
\draw[<-] ($(D.west)+(0,-2mm)$) -- node[port,below]{$e_S$} ($(S.east)+(0,-2mm)$);

\draw[->] ($(D.east)+(0, 2mm)$) -- node[port,above]{$f_R$} ($(R.west)+(0, 2mm)$);
\draw[<-] ($(D.east)+(0,-2mm)$) -- node[port,below]{$e_R$} ($(R.west)+(0,-2mm)$);

\draw[->] ($(D.south)+(-2mm,0)$) -- node[port,left] {$f_P$} ($(P.north)+(-2mm,0)$);
\draw[<-] ($(D.south)+( 2mm,0)$) -- node[port,right]{$e_P$} ($(P.north)+(2mm,0)$);
\end{tikzpicture}
\caption{Schematic representation of port-Hamiltonian systems}
\label{fig:moin}
\end{figure}
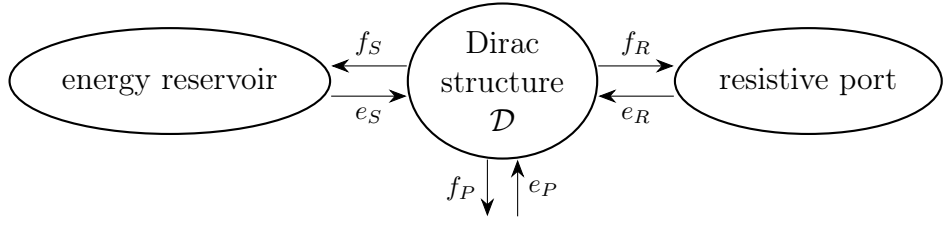

Port-Hamiltonian systems are modelled by Dirac structures, which describe the lossless routing of power within the system. The power is routed between the energy reservoir of the system, the resistive port, and the external port, which is diagrammatically represented in Figure~\ref{fig:moin}. The precise mathematical definitions port-Hamiltonian systems and their geometrical data are recalled in Section~\ref{sec:pH_repetition}. Structure preserving interconnection of port-Hamiltonian systems is diagrammatically described in Figure~\ref{fig:interconnection}, where we have suppressed possibly remaining external ports.
\begin{figure}[!ht]
\centering
\begin{tikzpicture}[
  >={Stealth[length=2.4mm]},
  block/.style={draw, thick, ellipse, align=center, inner sep=2pt,
                minimum height=14mm, minimum width=20mm},
  dirac/.style={draw, thick, ellipse, align=center, inner sep=2pt, 
  				minimum height=14mm, minimum width=20mm},
  port/.style={font=\footnotesize},
  ann/.style={font=\scriptsize, align=center},
  d/.style={4mm}                      
]

\node[dirac] (D) {Dirac\\structure\\$\mathcal{D}$};
\node[block, left=10mm of D]  (S) {energy reservoir};
\node[block, right=10mm of D] (R) {resistive port};
\node[dirac, below=7.5mm of D] (Dint) {Dirac\\structure\\$\mathcal{D}_{\mathrm{int}}$};
\node[dirac, below=7.5mm of Dint] (D2) {Dirac\\structure\\$\mathcal{D}'$};
\node[block, left=10mm of D2]  (S2) {energy reservoir};
\node[block, right=10mm of D2] (R2) {resistive port};

\draw[->] ($(D.west)+(0, 2mm)$) -- node[port,above]{$f_S$} ($(S.east)+(0, 2mm)$);
\draw[<-] ($(D.west)+(0,-2mm)$) -- node[port,below]{$e_S$} ($(S.east)+(0,-2mm)$);
\draw[->] ($(D2.west)+(0, 2mm)$) -- node[port,above]{$f_S'$} ($(S2.east)+(0, 2mm)$);
\draw[<-] ($(D2.west)+(0,-2mm)$) -- node[port,below]{$e_S'$} ($(S2.east)+(0,-2mm)$);

\draw[->] ($(D.east)+(0, 2mm)$) -- node[port,above]{$f_R$} ($(R.west)+(0, 2mm)$);
\draw[<-] ($(D.east)+(0,-2mm)$) -- node[port,below]{$e_R$} ($(R.west)+(0,-2mm)$);
\draw[->] ($(D2.east)+(0, 2mm)$) -- node[port,above]{$f_R'$} ($(R2.west)+(0, 2mm)$);
\draw[<-] ($(D2.east)+(0,-2mm)$) -- node[port,below]{$e_R'$} ($(R2.west)+(0,-2mm)$);

\draw[->] ($(D.south)+(-2mm,0)$) -- node[port,left] {$f_P$} ($(Dint.north)+(-2mm,0)$);
\draw[<-] ($(D.south)+( 2mm,0)$) -- node[port,right]{$e_P$} ($(Dint.north)+(2mm,0)$);
\draw[->] ($(Dint.south)+(-2mm,0)$) -- node[port,left] {$f_P'$} ($(D2.north)+(-2mm,0)$);
\draw[<-] ($(Dint.south)+( 2mm,0)$) -- node[port,right]{$e_P'$} ($(D2.north)+(2mm,0)$);
\end{tikzpicture}
\caption{Interconnection of port-Hamiltonian systems}
\label{fig:interconnection}
\end{figure}
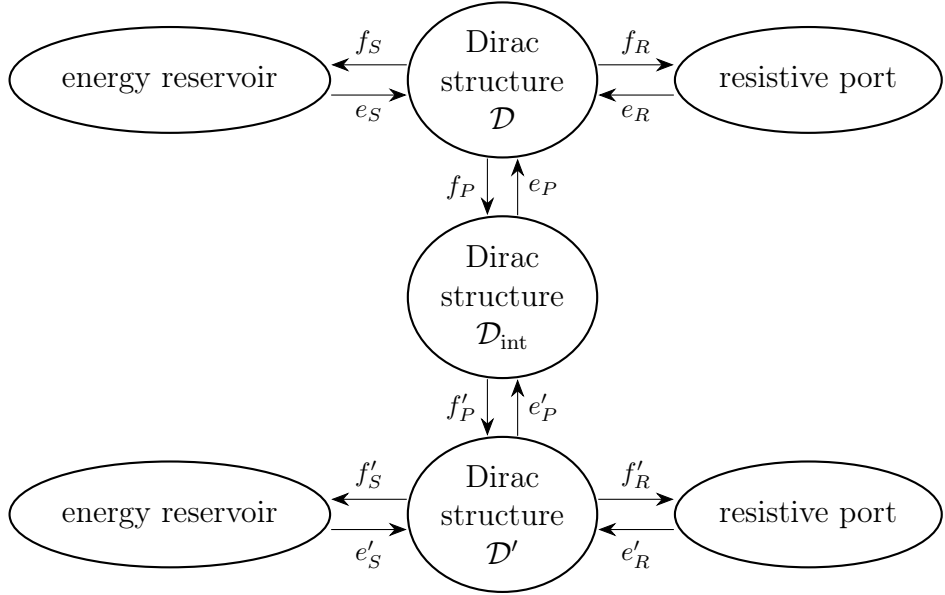
The structure-preservation of the interconnection is given by the geometric fact that the algebraic relation of the external port-variables by the Dirac structure $\mathcal{D}_{\mathrm{int}}$ is again a Dirac structure, cf.~\cite{Interconnection07,Interconnection18}, schematically represented in Figure~\ref{fig:interconnected_dirac_structure}.
\begin{figure}
\begin{tikzpicture}[
  >={Stealth[length=2.4mm]},
  dirac/.style={draw, thick, ellipse, align=center, inner sep=2pt, dashed},
  port/.style={font=\footnotesize},
  ann/.style={font=\scriptsize, align=center},
  d/.style={4mm}                      
]

\node[dirac] (D) {Dirac\\structure\\$\mathcal{D}$};
\node[dirac, right=7.5mm of D] (Dint) {Dirac\\structure\\$\mathcal{D}_{\mathrm{int}}$};
\node[dirac, right=7.5mm of Dint] (D2) {Dirac\\structure\\$\mathcal{D}'$};
\node[left=20mm of D] (X) {};
\node[right=20mm of D2] (Y) {};

\draw[thick] (Dint) ellipse [x radius=45.5mm, y radius=21mm];
\node[below=2.5mm of Dint, font=\bfseries] (DD) {$\mathcal{D}\Vert_{\mathcal{D}_{\mathrm{int}}}\mathcal{D}'$};
\node[above=2.5mm of Dint, font=\bfseries] (DD) {Dirac structure};

\draw[->] ($(D.west)+(0, 2mm)$) -- node[port,above]{$(f_S,f_R)$} ($(X.east)+(0, 2mm)$);
\draw[<-] ($(D.west)+(0,-2mm)$) -- node[port,below]{$(e_S,e_R)$} ($(X.east)+(0,-2mm)$);
\draw[->] ($(Dint.west)+(0, 2mm)$) -- node[port,above]{$f_P$} ($(D.east)+(0, 2mm)$);
\draw[<-] ($(Dint.west)+(0,-2mm)$) -- node[port,below]{$e_P$} ($(D.east)+(0,-2mm)$);
\draw[->] ($(D2.west)+(0, 2mm)$) -- node[port,above]{$f_P'$} ($(Dint.east)+(0, 2mm)$);
\draw[<-] ($(D2.west)+(0,-2mm)$) -- node[port,below]{$e_P'$} ($(Dint.east)+(0,-2mm)$);
\draw[->] ($(D2.east)+(0, 2mm)$) -- node[port,above]{$(f_S',f_R')$} ($(Y.west)+(0, 2mm)$);
\draw[<-] ($(D2.east)+(0,-2mm)$) -- node[port,below]{$(e_S',e_R')$} ($(Y.west)+(0,-2mm)$);
\end{tikzpicture}
\caption{Interconnected Dirac structure}
\label{fig:interconnected_dirac_structure}
\end{figure}
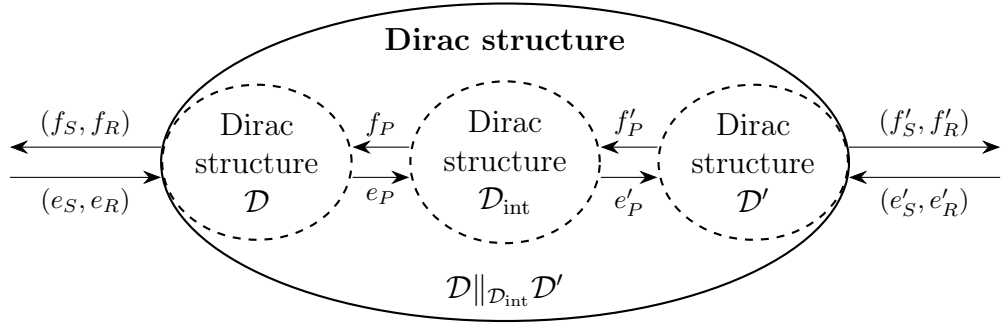
In the interconnected Dirac structure $\mathcal{D}\Vert_{\mathcal{D}_{\mathrm{int}}}\mathcal{D}'$, the external port-variables $(f_P,e_p)$ and $(f_P',e_P')$ do not appear explicitly, and the interconnected system is usually written in the variables $(f_S,f_S',e_S,e_S')$ as the port-Hamiltonian system sketched in Figure~\ref{fig:interconnected_pH}.
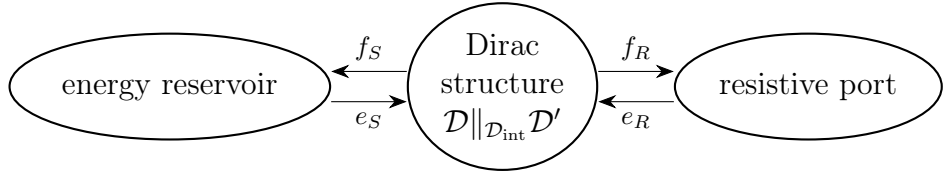
\begin{figure}[!ht]
\centering
\begin{tikzpicture}[
  >={Stealth[length=2.4mm]},
  block/.style={draw, thick, ellipse, align=center, inner sep=2pt,
                minimum height=14mm, minimum width=20mm},
  dirac/.style={draw, thick, ellipse, align=center, inner sep=2pt, 
  				minimum height=14mm, minimum width=20mm},
  port/.style={font=\footnotesize},
  ann/.style={font=\scriptsize, align=center},
  d/.style={4mm}                      
]

\node[dirac] (D) {Dirac\\structure\\$\mathcal{D}\Vert_{\mathcal{D}_{\mathrm{int}}}\mathcal{D}'$};

\node[block, left=10mm of D]  (S) {energy reservoir};
\node[block, right=10mm of D] (R) {resistive port};

\draw[->] ($(D.west)+(0, 2mm)$) -- node[port,above]{$f_S$} ($(S.east)+(0, 2mm)$);
\draw[<-] ($(D.west)+(0,-2mm)$) -- node[port,below]{$e_S$} ($(S.east)+(0,-2mm)$);

\draw[->] ($(D.east)+(0, 2mm)$) -- node[port,above]{$f_R$} ($(R.west)+(0, 2mm)$);
\draw[<-] ($(D.east)+(0,-2mm)$) -- node[port,below]{$e_R$} ($(R.west)+(0,-2mm)$);
\end{tikzpicture}
\caption{Interconnected port-Hamiltonian systems}
\label{fig:interconnected_pH}
\end{figure}
The equivalence of the systems represented by Figure~\ref{fig:interconnection} and Figure~\ref{fig:interconnected_pH} is expressed by the algebraic equivalence: For variables $(f_S,f_S',f_R,f_R',e_S,e_S',e_R,e_R')$, there exist pairs of port-variables $(f_P,e_p)$ and $(f_P',e_P')$ satisfying the relations sketched in Figure~\ref{fig:interconnection} if, and only if, $(f_S,f_S',f_R,f_R',e_S,e_S',e_R,e_R')$ satisfies those of Figure~\ref{fig:interconnected_pH}.

For dynamical systems, however, equivalence is usually not on the level of their variables alone, but instead expressed as a one-to-one correspondence of the trajectories of the systems, which are curves of the variables with certain regularity assumptions. Since port-Hamiltonian systems are dynamical systems, their interconnection should be studied from the dynamical systems' point of view. Then however, it is \textit{a priori} not guaranteed that, given a \textit{regular} curve of variables $(f_S,f_S',f_R,f_R',e_S,e_S',e_R,e_R')$ solving the \textit{interconnected system} of Figure~\ref{fig:interconnected_pH}, there are \textit{regular} curves of port-variables $(f_P,e_p)$ and $(f_P',e_P')$ so that the \textit{interconnection system} of Figure~\ref{fig:interconnection} is solved.

In this note, we demonstrate that there are port-Hamiltonian systems whose interconnected system is not behaviourally equivalent to the interconnection system, but has more solutions. We demonstrate that a necessary condition for the inequality of the trajectories is that the Dirac structures are non-constant. Since most examples of port-Hamiltonian systems in modelling and control are formulated with respect to constant Dirac structures, it is not surprising that the possible discrepancy has, to the best knowledge of the authors, not been observed earlier. Moreover, we derive a sufficient condition for the behavioural equivalence of the two systems, which is also a sufficient condition for the (geometric) interconnected Dirac structure to be a smooth Dirac structure, which further highlights the obscurity of our example. Using the sufficient condition as well as ``by foot'', we demonstrate that the port-Hamiltonian systems represented by ordinary differential equations without feed-through always interconnect cleanly.

The note is organised as follows. In section~\ref{sec:prelim}, we recall the necessary mathematical preliminaries, in particular port-Hamiltonian systems, weakly differentiable functions with values in manifolds, and the behavioural point of view on dynamical systems. Next, we define the classical and the weak behaviours of port-Hamiltonian systems, and recall the geometric definition of port-Hamiltonian systems with a sufficient, but not necessary, condition for this to be well-defined. This allows us to formulate the main problem: Is the behaviour of the system describing the interconnection of port-Hamiltonian systems identical to that of the interconnected system? An example shows that this is not the case. In the final section of the note, we demonstrate that the sufficient condition for well-definedness of the geometric interconnection also suffices for the two behaviours to be identical.

\section{Mathematical preliminaries}\label{sec:prelim}

In this section, we provide the precise definitions of the concepts sketched out in the introduction. In particular, we recall the geometric definition of port-Hamiltonian systems by means of the underlying geometric data. Here, we consider finite-dimensional, non-linear port-Hamiltonian systems. The non-linearity is expressed both in a non-linear Hamiltonian function, as well as in non-constant Dirac structures and the state-space having the structure of a general manifold. Here, we do assume some familiarity with the elementary notions of differential geometry, in particular manifolds and vector bundles. To keep the note somewhat self-contained and to give the link to port-Hamiltonian systems defined as descriptor systems in the Euclidean coordinate space, we give the local coordinate representations. For simplicity of exposition, we shall assume that the geometric data are smooth. 

In the final part of the section, we present the precise mathematical language which allows us to talk about all local trajectories of dynamical systems simultaneously: the behavioural language introduced by Willems~\cite{Polderman}. To treat classical as well as weak solutions, we precede this section by giving a brief background in the Sobolev spaces of curves with values in general smooth manifolds. The local nature of the solutions is expressed in the language of sheaves over a category of compact intervals, cf.~\cite{SchuSPivVasi20}. We stress that we merely utilise the language of sheaves to express the inherent properties of solutions of differential-algebraic systems. In particular, we do not go into any depth of sheaf theory, and knowledge of this is not necessary, nor assumed, for the understanding of this note.

\subsection{Port-Hamiltonian systems}\label{sec:pH_repetition}

The geometric data of port-Hamiltonian systems are Dirac structures and Lagrangian subbundles. The former were introduced by Courant and Weinstein~\cite{CourWein88,Cour90} and Dorfman~\cite{Dorf87,Dorf93} as a unified description of the modelling of Hamiltonian systems defined with respect to almost Poisson or presymplectic forms and in the presence of non-holonomic constraints; the latter are well-known in symplectic geometry, cf.~\cite{LibeMarl87}.

\begin{definition}\label{def:underlying_geometry}
Let $E\to M$ be a vector bundle with dual bundle $E^*$. The Whitney sum $E\oplus E^*$ is equipped with the pseudo-Euclidean metric
\begin{align*}
\llangle\cdot,\cdot\rrangle: (E\oplus E^*)\otimes_\mathbb{R}(E\oplus E^*)\to\mathbb{R},\qquad \llangle (f,e),(f',e')\rrangle := \langle e',f\rangle+\langle e,f'\rangle
\end{align*}
and the symplectic form
\begin{align*}
\Omega: (E\oplus E^*)\otimes_\mathbb{R}(E\oplus E^*)\to\mathbb{R},\qquad \Omega\big((f,e),(f',e')\big) := \langle e',f\rangle-\langle e,f'\rangle.
\end{align*}
\begin{enumerate}[(i)]
\item A \textit{Dirac structure} is a subvector bundle $\mathcal{D}\leq E\oplus E^*$ which is isotropic with respect to $\llangle\cdot,\cdot\rrangle$ and has rank $\rk\mathcal{D} = \rk E$.
\item A \textit{Lagrangian subbundle} is a subvector bundle $\mathcal{L}\leq E\oplus E^*$ which is isotropic with respect to $\Omega$ and has rank $\rk\mathcal{D} = \rk E$.
\item A subvector bundle $\mathcal{E}\leq E\oplus E^*$ is \textit{non-negative} if, and only if, $\llangle\cdot,\cdot\rrangle\vert_{\mathcal{L}}$ is an Euclidean semimetric.
\item A \textit{constant} Dirac structure (Lagrangian subbundle) is the pullback of a Dirac structure (Lagrangian subbundle) on a vector bundle $E\to\set{\mathrm{pt}}$ under the constant map $M\to\set{\mathrm{pt}}$.
\end{enumerate}
\end{definition}

Strictly speaking, the term \textit{Dirac structure} for the objects described in Definition~\ref{def:underlying_geometry}\,(i) is imprecise. Both Courant and Dorfman introduce for the case that $E = \mathcal{T}M$ is the tangent bundle of a manifold, in which case $E\oplus E^* = \mathcal{T}M\oplus\mathcal{T}^*M =: \mathbb{T}M$ is called extended tangent bundle or Pontryagin bundle, a bracket on the (local) sections of $\mathbb{T}M$. A Dirac structure is then an object as described in Definition~\ref{def:underlying_geometry}\,(i) which is additionally \textit{integrable}, i.e. whose (local) sections are additionally closed under that bracket. For port-Hamiltonian systems, which are generally not defined on a Pontryagin bundle, a canonical Courant or Dorfman bracket does not exist and an integrability condition is not universally known, cf.~\cite{Merk09} for a discussion of potential notions of integrability. Therefore, Dirac structures in the context of port-Hamiltonian systems are generally understood as non-integrable. 

Locally, it is well-known that Dirac structures and Lagrangian subbundles are given by matrix fields as follows, cf.~\cite{Cour90,SchaftJeltsema14}.

\begin{remark}\label{rem:local_reps}
Let $E = M\times\mathbb{R}^n$ be a trivial vector bundle. Then, $E^*$ is identified with $E$ by the fibrewise Euclidean scalar product, so that the duality product in the canonical pseudo-Euclidean metric and symplectic form is the fibrewise Euclidean scalar product. Dirac structures and Lagrangian submanifolds are then subbundles of the trivial vector bundle $M\times\mathbb{R}^{2n}$. Assume that $\mathcal{D}\leq E\oplus E^*$ is a subbundle of rank $n$ which is globally trivialisable. Then, there exist matrix fields $E,F\in\mathcal{C}^\infty(M,\mathbb{R}^{n\times n})$ so that $\rk [F(\cdot)^\top,E(\cdot)^\top] = n$ and
\begin{align}\label{eq:image_representation}
\mathcal{D} = \set{(x,F(x)\lambda,E(x)\lambda)\,\big\vert\,x\in M, \lambda\in\mathbb{R}^n}.
\end{align}
Then, $\mathcal{D}$ is isotropic with respect to the canonical pseudo-Euclidean metric if, and only if, $E(x)^\top F(x) + F(x)^\top E(x) = 0$ for all $x\in M$, and isotropic with respect to the canonical symplectic form if, and only if, $E(x)^\top F(x) - F(x)^\top E(x) = 0$; if additionally $E(\cdot)^\top F(\cdot)$ is pointwise positive semidefinite, then $\mathcal{D}$ is non-negative (and \textit{vice versa}).  Analogously, there exists matrix fields $\widetilde{E},\widetilde{F}\in\mathcal{C}^\infty(M,\mathbb{R}^{n\times n})$ so that $\rk [\widetilde{E}(\cdot),\widetilde{F}(\cdot)]\equiv n$ and
\begin{align}\label{eq:kernel_representation}
\mathcal{D} = \set{(x,X,\alpha)\,\big\vert\,x\in M, X,\alpha\in\mathbb{R}^n~\mathrm{with}~\widetilde{E}(x)X+\widetilde{F}(x)\alpha = 0}.
\end{align}
It can be shown that $\mathcal{D}$ is a Dirac structure if, and only if, $F(\cdot)E(\cdot)^\top + E(\cdot)F(\cdot)^\top \equiv 0$, and Lagrangian if, and only if, $F(\cdot)E(\cdot)^\top - E(\cdot)F(\cdot)^\top \equiv 0$; non-negativity is characterised by (pointwise) non-negativity of $F(\cdot)E(\cdot)^\top$. The representation~\eqref{eq:image_representation} is the \textit{image representation} and~\eqref{eq:kernel_representation} is the \textit{kernel representation}.

Of course, not every vector bundle is globally trivial. In the general case, $E$ and $\mathcal{D}$ are locally trivialisable so that around each $x_0\in M$ there exists an open neighbourhood $U\subseteq M$ of $x_0$ so that $E\vert_{U}$ and $\mathcal{D}\vert_U$ are trivialisable and admit kernel and image representations.
\end{remark}

Before we continue with the definition of port-Hamiltonian systems, we give some elementary examples of Dirac structures and Lagrangian subbundles, cf.~\cite{Cour90}.

\begin{example}\label{ex:fun_Dirac_structures}
\begin{enumerate}[(i)]
\item Let $E\to M$ be any vector bundle, and consider an antisymmetric tensor field $\omega\in\Gamma(E^*\wedge E^*)$, where $\Gamma(\cdot)$ denotes the space of global sections of a vector bundle. Then,
\begin{align*}
\mathcal{D}_\omega := \set{(X,\alpha)\,\big\vert\,X\in E_x, x\in M, \forall Y\in E_x: \langle\omega,X\otimes Y\rangle = \langle\alpha,X\rangle},
\end{align*}
where $\langle\cdot,\cdot\rangle$ denotes the respective duality products, is a Dirac structure. Analogously, if $\pi\in\Gamma(E\wedge E)$, then
\begin{align*}
\mathcal{D}_\pi := \set{(X,\alpha)\,\big\vert\,\alpha\in E^*_x, x\in M, \forall \beta\in E^*_x: \langle\alpha\otimes \beta,\pi\rangle = \langle\alpha,X\rangle}
\end{align*}
is a Dirac structure. In case that the symmetric instead of the antisymmetric tensor product is considered, i.e. $\omega\in\Gamma(E^*\vee E^*)$ and $\pi\in\Gamma(E\vee E)$, $\mathcal{D}_\omega$ and $\mathcal{D}_\pi$ are Lagrangian subbundles. In particular, $E\oplus 0$ and $0\oplus E^*$ are both Dirac structures and Lagrangian subbundles.
\item In case of a constant vector bundle $M\times\mathbb{R}^n$, the Dirac structures described in (i) are given by (pointwise) antisymmetric matrix fields $J,\Omega\in\mathcal{C}^\infty(M,\mathbb{R}^{n\times n})$ as
\begin{align*}
\mathcal{D}_\Omega & := \set{\big(x,X,\Omega(x)X\big)\,\big\vert\,x\in M, X\in\mathbb{R}^n},\\
\mathcal{D}_J & := \set{\big(x,J(x)\alpha,\alpha\big)\,\big\vert\,x\in M, \alpha\in\mathbb{R}^n},
\end{align*}
and the Lagrangian subbundles are given by (pointwise) symmetric matrix fields. These are the standard examples of Dirac structures.
\item An example of a non-trivialisable Dirac structure whose encompassing Pontryagin bundle is trivialisable is given as follows: Let $n\in2\mathbb{N}^*$ be any positive even integer. The Pontryagin bundle of the $n$-dimensional sphere is trivialisable, cf.~\cite[Corollary 2.4]{EtayGoNiSant26}. By the hairy ball theorem and (i), $\mathcal{T}\mathcal{S}^n$ is a non-trivialisable Dirac structure.
\end{enumerate}
\end{example}

With the precise definition of the geometric structures of port-Hamiltonian systems at hand, we may finally recall the precise definition of port-Hamiltonian systems. The origin of this definition can be traced back to~\cite{DalsvdS98} with the first appearance of a somewhat fully geometric definition (i.e. external and resistive ports are modelled by general vector bundles) in~\cite{Merk09}, see also the~\cite{SchaftJeltsema14}.

\begin{definition}[cf.~\cite{SchaftJeltsema14,Merk09}]\label{def:port-Hamiltonian}
A \textit{port-Hamiltonian system} is the dynamical system
\begin{equation}\label{eq:port-Hamiltonian}
\begin{aligned}
\big(\tfrac{\mathrm{d}}{\mathrm{d}t}x,f,-y,\mathrm{d}H(x),e,u\big) & \in\mathcal{D}_x,\\
(f,e) & \in\mathcal{R}_x,
\end{aligned}
\end{equation}
with the underlying data
\begin{enumerate}[(i)]
\item vector bundles $E_r\to M$ and $E_p\to M$ with dual bundles $E_r^*$ and $E_p^*$, respectively,
\item a Dirac subbundle $\mathcal{D}\leq \big(\mathcal{T}M\oplus E_r\oplus E_p\big)\oplus\big(\mathcal{T}^*M\oplus E_r^*\oplus E_p^*\big)$,
\item a non-negative Lagrangian subbundle $\mathcal{R}\leq E_r\oplus E_r^*$,
\item a Hamiltonian function $H\in\mathcal{C}^\infty(M)$.
\end{enumerate}
We write $(E_r\oplus E_p\to M,\mathcal{D},\mathcal{R},H)$ for the tupel of the data of the port-Hamiltonian system~\eqref{eq:port-Hamiltonian}. The variables $(f,e)$ and $(y,u)$ are the pairs of resistive and external port-variables, respectively.
\end{definition}

We illustrate the definition by giving the representations of port-Hamiltonian systems in local coordinates.

\begin{remark}
Consider a port-Hamiltonian system with data $(E_r\oplus E_p\to M,\mathcal{D},\mathcal{R},H)$, and let $x_0\in M$. There exists an open neighbourhood of $U\subseteq M$ of $x_0$ admitting a chart $U\simto V\subseteq\mathbb{R}^n$ so that $E_r\vert_U\simeq V\times\mathbb{R}^{m_r}$, $E_p\vert_U\simeq V\times\mathbb{R}^{m_p}$, $\mathcal{D}\vert_U\simeq V\times\mathbb{R}^{n+m_r+m_p}$ and $\mathcal{R}\vert_U\simeq V\times\mathbb{R}^{m_r}$ are trivialisable. Then, there are matrix fields $E,F\in\mathcal{C}^\infty\big(V,\mathbb{R}^{(n+m_r+m_p)\times(n+m_r+m_p)}\big)$ so that
\begin{align*}
\mathcal{D}\vert_U\simeq \set{(x,X,f,y,\alpha,e,u)\in V\times\mathbb{R}^{2(n+m_r+m_p)}\,\left\vert\,E(x)\begin{pmatrix}
X\\f\\y
\end{pmatrix} = F(x)\begin{pmatrix}
\alpha\\e\\u
\end{pmatrix}\right.}
\end{align*}
and matrix fields $R_1,R_2\in\mathcal{C}^\infty\big(V,\mathbb{R}^{m_r\times m_r}\big)$ with
\begin{align*}
\mathcal{R}\vert_U\simeq\set{(x,R_1(x)\lambda,R_2(x)\lambda)\,\big\vert\,x\in V,\lambda\in\mathbb{R}^{m_r}}.
\end{align*}
Then, the dynamical system
\begin{equation*}
\begin{aligned}
\big(\tfrac{\mathrm{d}}{\mathrm{d}t}x,f,-y,\mathrm{d}H(x),e,u\big) & \in\mathcal{D}_x,\\
(f,e) & \in\mathcal{R}_x,\\
x & \in U
\end{aligned}
\end{equation*}
is \textit{equivalent} to the non-linear differential-algebraic equation
\begin{align}\label{eq:local_pHs}
E(x)\begin{pmatrix}
\tfrac{\mathrm{d}}{\mathrm{d}t}x\\R_1(x)\lambda\\-y
\end{pmatrix} = F(x)\begin{pmatrix}
\nabla H(x)\\R_2(x)\lambda\\u
\end{pmatrix}.
\end{align}
The matrix fields $E$ and $F$ have the properties
\begin{align}\label{eq:local_properties}
\rk [E(\cdot),F(\cdot)] \equiv n+m_r+m_p,\qquad E(\cdot)F(\cdot)^\top+F(\cdot)E(\cdot)^\top\equiv 0;
\end{align}
analogously, $R_1$ and $R_2$ satisfy
\begin{align}\label{eq:local_properties_2}
\rk\begin{bmatrix}
R_1(\cdot)\\R_2(\cdot)
\end{bmatrix}\equiv m_r,\hspace*{3mm} R_1(\cdot)^\top R_2(\cdot)-R_2(\cdot)^\top R_1(\cdot)\equiv 0,\hspace*{3mm} R_1(\cdot)^\top R_2(\cdot)\geq 0.
\end{align}
Conversely, every dynamical system~\eqref{eq:local_pHs} whose matrix fields satisfy the conditions~\eqref{eq:local_properties} and~\eqref{eq:local_properties_2} is the local representation of a port-Hamiltonian system.
\end{remark}

\begin{example}\label{ex:pH_ODE_system}
The well-known example of port-Hamiltonian systems are port-Hamiltonian ODEs given by the dynamical system
\begin{equation}\label{eq:pH_ODE_system}
\begin{aligned}
\tfrac{\mathrm{d}}{\mathrm{d}t} & = \big(J(x)-R(x)\big)\nabla H(x) + B(x)u\\
y & = B(x)^\top\nabla H(x)
\end{aligned}
\end{equation}
with $J,R\in\mathcal{C}^\infty\big(U,\mathbb{R}^{n\times n}\big)$, $B\in\mathcal{C}^\infty\big(U,\mathbb{R}^{n\times m}\big)$, $H\in\mathcal{C}^\infty(U)$, $U\subseteq\mathbb{R}^n$ open and non-empty, so that $J(\cdot)+J(\cdot)^\top\equiv 0$, $R(\cdot)-R(\cdot)\equiv 0$ and $R(\cdot)\geq 0$. To see that~\eqref{eq:pH_ODE_system} is indeed \textit{equivalent} to a port-Hamiltonian system, consider the Dirac structure
\begin{align*}
\mathcal{D} := \set{\big(x,J(x)\alpha+e+B(x)u,-\alpha,-B(x)^\top\alpha,\alpha,e,u\big)\,\left\vert\,x\in U, \begin{pmatrix}
\alpha\\e\\u
\end{pmatrix}\in\mathbb{R}^{N}\right.},
\end{align*}
where $N = 2n+m$ and $\mathcal{T}U\simeq U\times\mathbb{R}^n$ are identified by means of the canonical basis, and consider the non-negative Lagrangian subbundle
\begin{align*}
\mathcal{R} := \set{\big(x,f,R(x)f\big)\,\big\vert\,x\in U, f\in\mathbb{R}^n}.
\end{align*}
Then, the dynamical system~\eqref{eq:pH_ODE_system} is \textit{equivalent} to the port-Hamiltonian system
\begin{align*}
\big(\tfrac{\mathrm{d}}{\mathrm{d}t}x,f,-y,\nabla H(x),e,u\big) & \in\mathcal{D}_x,\\
(f,e) & \in\mathcal{R}_x.
\end{align*}
Evidently, this representation is non-unique.
\end{example}

\subsection{Weakly differentiable curves}\label{sec:Sobolev}

In this section, we recall the basic definitions and properties of weakly differentiable curves with values in manifolds. The classical definition for curves with values in the Euclidean space is the following.

\begin{definition}
Let $I\subseteq\mathbb{R}$ be open. A curve $x:I\to\mathbb{R}^n$ is weakly differentiable with weak derivative $\dot{x}:I\to\mathbb{R}^n$ if, and only if, for each test function $\varphi\in\mathcal{C}^\infty_c(I,\mathbb{R}^n)$, i.e. $\varphi\in\mathcal{C}^\infty(I,\mathbb{R}^n)$ with compact $I\setminus\varphi^{-1}(\set{0})\subseteq [t_0,t_1]$,
\begin{align*}
\int_I \varphi(\tau)^\top \dot{x}(\tau)\,\mathrm{d}\,\tau + \int_I \left(\tfrac{\mathrm{d}}{\mathrm{d}t}\varphi(\tau)\right)^\top x(\tau)\,\mathrm{d}\,\tau.
\end{align*}
\end{definition}

Convent and van Schaftingen~\cite{ConvScha16} propose to define the weakly differentiable functions as the ``pullback'' of vector valued weakly differentiable curves by means of local coordinates.

\begin{definition}[{cf.~\cite[Definition 1.1 and Definition 1.3]{ConvScha16}}]\label{def:Sobolev_on_manifolds}
Let $M$ be a smooth manifold with atlas $(U,\varphi_U)_{U\in\mathfrak{U}}$. A \textit{continuous} curve $x:I\to M$, $I\subseteq\mathbb{R}$ open, is \textit{weakly differentiable} with weak derivative $X:I\to\mathcal{T}M$ if, and only if, for each $\in\mathfrak{U}$, $x^*\varphi_U:x^{-1}(U)\to \mathbb{R}^n$ is weakly differentiable with weak derivative $\mathrm{d}\varphi_UX$, where $x^*\varphi_U := \varphi_U\circ x$ denotes the pullback of $\varphi$ by $x$.
\end{definition}

By means of a decomposition of unity on the preimage $I$, it can be verified that Definition~\ref{def:Sobolev_on_manifolds} is independent of the atlas, and hence weakly differentiable curves are well-defined. However, continuity of $x$ is an essential ingredient for this property. This is not problematic in the context of this note, since continuity of weak solutions is usually assumed. Moreover, Definition~\ref{def:Sobolev_on_manifolds} is (for continuous curves) equivalent to~\cite[Definition 1.1 and Definition 1.3]{ConvScha16}, where the composition with all scalar valued functions with compact support is assumed to be weakly differentiable. Our definition has the advantage that the extension to curves with values in Banach manifolds is straightforward. 

An alternative definition of weakly differentiable curves considers an isometric embedding of $M$, which is equipped with a Riemannian metric, into an Euclidean space, cf.~\cite{Hajl09} and~\cite{ConvScha16} for an extensive reference to the existing literature. It can be shown, that this definition is equivalent to Definition~\ref{def:Sobolev_on_manifolds}, cf.~\cite[Proposition 2.7]{ConvScha16}. For higher orders of differentiability (which are defined as usual, i.e. a curve is $k+1$ times weakly differentiable if, and only if, it is $k$ times weakly differentiable and the $k$-th weak derivative is weakly differentiable, in which case its weak derivative is the $(k+1)$th weak derivative), this equivalence does, in general, not hold true, see the discussion in~\cite{ConvSchaf19}. In this note, we are mostly concerned with $k = 1$, which is the usual minimal assumption for weak solutions of differential equations, and shall therefore ignore this sublety.

The definition of Sobolev spaces of \textit{integrable} functions utilises Lebesgue spaces with values in vector bundles, defined as follows.

\begin{definition}[{cf.~\cite[\S 8]{Pala68}}] \label{def:nonlinear_integrability}
Let $E\to M$ be a vector bundle, and let $g\in\Gamma(E^*\otimes E^*)$ be a vector bundle metric. A measurable curve $x:I\to E$, $I\subseteq\mathbb{R}$ an open or compact set, is in $\mathcal{L}^p_g(I,E)$, $p\in [1,\infty]$ if, and only if, $\sqrt{g(x,x)}\in \mathcal{L}^p(I,\mathbb{R})$.
\end{definition}

Evidently, Definition~\ref{def:nonlinear_integrability} depends on the metric, since not all vector bundle metrics are equivalent. Later, we are most interested in curves defined on a compact interval with continuous shadow, in which case the dependency on the metric vanishes by the following elementary result.

\begin{lemma}\label{lem:independence_of_metric}
Let $E\to M$ be a vector bundle, and let $g\in\Gamma(E^*\otimes E^*)$ and $\widetilde{g}\in\Gamma(E^*\otimes E^*)$ be vector bundle metrics.
\begin{enumerate}[(i)]
\item If $M$ is compact, then $g$ and $\widetilde{g}$ are equivalent.
\item For each compact interval $I$ and each curve $x:I\to E$ with continuous shadow, $x\in \mathcal{L}^p_g(I,E)$ if, and only if, $x\in \mathcal{L}^p_{\widetilde{g}}(I,E)$.
\end{enumerate}
\end{lemma}
\begin{proof}
Both $g$ and $\widetilde{g}$ correspond to vector bundle isomorphisms $\gamma:E\simto E^*$ and $\widetilde{\gamma}:E\simto E^*$ so that $g(e,e') = \langle\gamma(e),e'\rangle$ and $\widetilde{g}(e,e') = \langle\widetilde{\gamma}(e),e'\rangle$, respectively. Therefore we obtain for compact $M$, for all $e\in E$,
\begin{align*}
g(e,e) = \langle\widetilde\gamma\big((\gamma^*\widetilde{\gamma}^{-1})(e)\big),e\rangle = \widetilde{g}\big((\gamma^*\widetilde{\gamma}^{-1})(e),e)\leq  \max_{x\in M}\norm{(\gamma^*\widetilde{\gamma}^{-1})_x}\widetilde{g}(e,e)
\end{align*}
and, analogously, $\widetilde g(e,e)\leq \max_{x\in M}\norm{(\widetilde\gamma^*{\gamma}^{-1})_x}g(e,e)$. This demonstrates that $g$ and $\widetilde{g}$ are equivalent vector bundle metrics. Likewise, we have for each (measureable) curve $\xi:I\to E$ with continuous shadow $x$,
\begin{align*}
\left({\max_{t\in I}\norm{(\widetilde\gamma^*{\gamma}^{-1})_{x(t)}}}\right)^{-1}\sqrt{\widetilde g(\xi,\xi)} \leq {g(\xi,\xi)}\leq {\max_{t\in I}\norm{(\gamma^*\widetilde{\gamma}^{-1})_{x(t)}}}{\widetilde g(\xi,\xi)}
\end{align*}
\end{proof}

We have now recalled the definitions of weakly differentiable curves with values on a manifold, and of integrable curves with values on a vector bundle. This enables us to formulate, next to the classical, also the weak solutions of port-Hamiltonian systems. All solutions are collected in the behaviour, whose structure reflects the natural properties of solutions.

\subsection{Behaviours and sheaves}\label{sec:behaviours}

The behavioural point of view on dynamical systems was introduced by Willems, and is the basis of the monograph~\cite{Polderman}. A dynamical system in the behavioural picture is given as a triple $(\mathbb{T},\Sigma,\mathfrak{B})$ of a set $\mathbb{T}$, the time axis, a set $\Sigma$, the signal space, and a set $\mathfrak{B}\subseteq\Sigma^\mathbb{T}$, i.e. $\mathfrak{B}$ is a set of curves. In most practical applications, $\mathbb{T}$ is $\mathbb{R}$ or $\mathbb{R}_{\geq 0}$ for continuous time systems and $\mathbb{Z}$ or $\mathbb{N}$  for discrete time systems. The drawback of this approach is that elements of the behaviour share the same time axis, so phenomena such as finite escape times are not straightforward to represent in the theory. The second drawback is that local (in time) properties of solutions of dynamical systems are cumbersome, since all solutions are global. A more convenient description of behaviours is given in the language of \textit{sheaves} over a category of non-degenerate compact intervals, cf.~\cite{SchuSPivVasi20} for a similar treatment. In this note, we do not need the full mathematical apparatus of sheaf theory, for which we refer the reader to~\cite{Bred97,John02,MaLaMoer92}; instead, we borrow the terminology.

\begin{definition}[cf.~\cite{Bock26}]
An \textbf{Int}-sheaf $\mathfrak{F}$ is given by the data
\begin{enumerate}[(i)]
\item a set $\mathfrak{F}([a,b])$ for each $a<b\in\mathbb{R}$
\item a function $\mathfrak{F}([c,d]\subseteq[a,b]):\mathfrak{F}([a,b])\to\mathfrak{F}([c,d])$ for each $a\leq b<c\leq d$ 
\end{enumerate}
with the properties
\begin{enumerate}[(a)]
\item $\mathfrak{F}([a,b]\subseteq[a,b]) = \mathrm{id}_{\mathfrak{F}([a,b])}$
\item $\mathfrak{F}([e,f]\subseteq[c,d])\circ \mathfrak{F}([c,d]\subseteq[a,b]) = \mathfrak{F}([e,f]\subseteq[a,b])$
\item If $\mathfrak{F}([c,d]\subseteq[a,d])(x) = \mathfrak{F}([c,d]\subseteq[c,b])(y)$, then there exists a unique $z\in\mathfrak{F}([a,b])$ with $\mathfrak{F}([a,d]\subseteq [a,b])(z) = x$ and $\mathfrak{F}([c,b]\subseteq[a,b])(z) = y$.
\end{enumerate}
for all $a\leq c\leq e<f\leq d\leq b$ and $x\in\mathfrak{F}([a,d])$ and $y\in\mathfrak{F}([c,b])$. For simplicity, we adopt the notion
\begin{align*}
x\vert_{[c,d]} := \mathfrak{F}([c,d]\subseteq[a,b])(x)
\end{align*}
for the restriction morphisms $\mathfrak{F}([c,d]\subseteq[a,b])$.
\end{definition}

In this note, we are mainly interested in three \textbf{Int}-sheaves, the classical differentiable or continuous weakly differentiable curves with values in a manifold, and the vector bundle values Lebesgue spaces with continuous shadow. All those are collections of particular functions, whose restriction morphisms are the usual restrictions of functions; such sheaves shall be called sheaves of functions. The definition of the classical differentiable curves on compact intervals warrants some attention. One possibility to define these in the sheaf-theoretical spirit is to take germs of differentiable curves defined on an open neighbourhood of the compact interval. The alternative considers continuous curves admitting a differentiable extension onto an open neighbourhood. In this note, we consider the latter for simplicity.

\begin{lemma}
Let $M$ be a manifold and $k\in\mathbb{N}\cup\set{\infty}$. Define, for each $t_0<t_1\in\mathbb{R}$, 
\begin{align*}
\mathcal{C}^k([t_0,t_1],M) := \set{x\in\mathcal{C}([t_0,t_1],M)\,\left\vert\,\begin{array}{r}\exists \varepsilon>0\exists \overline{x}\in\mathcal{C}^k\big((t_0-\varepsilon,t_1+\varepsilon),M\big)\hspace*{3mm}\\: x = \overline{x}\vert_{[t_0,t_1]}
\end{array}\right.\!}.
\end{align*}
$\mathcal{C}^k(\cdot,M)$ is an \textbf{Int}-sheaf of functions.
\end{lemma}
\begin{proof}
The restriction morphisms are well-defined and satisfy the conditions (a) and (b). Let $a\leq c<d\leq b$, and let $x\in\mathcal{C}^k([a,d],M)$ and $y\in\mathcal{C}^k([c,b],M)$ with $[x]\vert_{[c,d]} = [y]\vert_{[c,d]}$. Without loss of generality, there are extensions $\overline{x}\in\mathcal{C}^k\big((a-\varepsilon,d+\varepsilon),M\big)$ and $\overline{y}\in\mathcal{C}^k\big((c-\varepsilon,b+\varepsilon),M\big)$ for some $\varepsilon>0$ so that $\overline{x}\vert_{[a,d]} = x$ and $\overline{y}\vert_{[c,b]} = y$. In particular, the function
\begin{align*}
\overline{z}: (a-\varepsilon,b+\varepsilon)\to M,\qquad t\mapsto\begin{cases}
x(t), & t\in (a-\varepsilon,d),\\
y(t), & t\in (c,b+\varepsilon)
\end{cases}
\end{align*}
is the unique well-defined continuous function with $x = \overline{y}\vert_{[a,d]}$ and $y = \overline{z}\vert_{[c,b]}$. For $k>0$, $\overline{z}$ is $k$ times continuously differentiable since both $x$ and $y$ are continuously differentiable on $(a-\varepsilon,d)$ and $(c,b+\varepsilon)$, which have non-trivial overlap $(c,d)$. Therefore, $z := \overline{z}\vert_{[a,b]}\in\mathcal{C}^k([a,b],M)$ is the unique function with $z\vert_{[a,d]} = x$ and $z\vert_{[c,b]} = y$. This shows that $\mathcal{C}^k(\cdot,M)$ is an \textbf{Int}-sheaf.
\end{proof}

\begin{remark}
Since a function defined on an open interval is differentiable in a point if, and only if, it is differentiable from the left and from the right and both directional derivatives coincide, the differential operators
\begin{align*}
\tfrac{\mathrm{d}}{\mathrm{d}t}:\mathcal{C}^k(\cdot,M)\to\mathcal{C}^{k-1}(\cdot,\mathcal{T}M),\qquad \mathcal{C}^k([t_0,t_1],M)\ni\tfrac{\mathrm{d}}{\mathrm{d}t}x := \left.\left(\tfrac{\mathrm{d}}{\mathrm{d}t}\overline{x}\right)\right\vert_{[t_0,t_1]}
\end{align*}
where $\overline{x}$ is any $\mathcal{C}^k$ extension of $x$ and the symbols ``$\infty-1$'' and ``$\infty$'' are identified, are well-defined. An equivalent definition is
\begin{align*}
\tfrac{\mathrm{d}}{\mathrm{d}t}x(t) = \begin{cases}
\lim_{h\downarrow 0}\frac{x(t_0+h)-x(t_0)}{h}, & t=t_0,\\
\tfrac{\mathrm{d}}{\mathrm{d}t}\big(x\vert_{(t_0,t_1)}\big)(t), & t\in(t_0,t_1),\\
\lim_{h\downarrow 0}\frac{x(t_1-h)-x(t_1)}{h}, & t=t_1. 
\end{cases}
\end{align*}
\end{remark}

The definition of the continous weakly differentiable curves and the Lebesgue spaces of curves (with continuous shadow) is straightforward. We briefly verify that these have indeed the structure of an \textbf{Int}-sheaf.

\begin{lemma}\label{lem:weak_sheaves}
Let $\pi:E\to M$ be a vector bundle and $p\in[1,\infty]$. Choose a vector bundle metric $g$ and define, for each $t_0<t_1\in\mathbb{R}$, 
\begin{align*}
\mathcal{L}_c^p([t_0,t_1],E) := \set{\xi:[t_0,t_1]\to E\,\left\vert\,\begin{array}{l}
\xi^*\pi\in\mathcal{C}([t_0,t_1],M),\\
\sqrt{g(\xi,\xi)}\in \mathcal{L}^p([t_0,t_1],\mathbb{R})
\end{array}\right.\!}
\end{align*}
and
\begin{align*}
\big(\mathcal{C}\cap W^{1,p}\big)([t_0,t_1],M) := \set{x\in\mathcal{C}([t_0,t_1],M)\,\left\vert\,\begin{array}{l}
x~\text{weakly~differentiable}\\
\text{with~weak~derivative}\\X\in \mathcal{L}_c^p([t_0,t_1],\mathcal{T}M)
\end{array}\right.\!
}.
\end{align*}
Both $\mathcal{L}_c^p(\cdot,E)$ and $\big(\mathcal{C}\cap W^{1,p}\big)(\cdot,M)$ are \textbf{Int}-sheaves of functions.
\end{lemma}
\begin{proof}
By Lemma~\ref{lem:independence_of_metric}, $\mathcal{L}_c^p(\cdot,E)$ is well-defined. Given $t_0<t_0'<t_1'<t_1\in\mathbb{R}$, $\xi\in\mathcal{L}_c^p([t_0,t_1'],E)$ and $\xi'\in\mathcal{L}_c^p([t_0',t_1],E)$ with $\xi\vert_{[t_0',t_1']} = \xi'\vert_{[t_0',t_1']}$, the function
\begin{align*}
\zeta: [t_0,t_1]\to E,\qquad t\mapsto\begin{cases}
\xi(t), & t\in[t_0,t_1']\\
\xi'(t), & t\in[t_0',t_1]
\end{cases}
\end{align*}
is the unique well-defined function with $\zeta^*\pi\in\mathcal{C}([t_0,t_1],M)$ and $\zeta\vert_{[t_0,t_1']} = \xi$ and $\zeta\vert_{[t_0',t_1]} = \xi'$. With
\begin{align*}
\sqrt{g(\zeta,\zeta)} = \chi_{[t_0,t_1']}\sqrt{g(\xi,\xi)}+\chi_{[t_0',t_1]}\sqrt{g(\xi',\xi')},
\end{align*}
where $\chi_\cdot$ denotes the indicator function of a set and the partially defined functions $\sqrt{g(\xi,\xi)},\sqrt{g(\xi',\xi')}$ are extended by zero, it is seen that $\sqrt{g(\zeta,\zeta)}\in\mathcal{L}^p([t_0,t_1],\mathbb{R})$. Therefore, $\zeta\in\mathcal{L}_c^p([t_0,t_1],E)$ and $\mathcal{L}_c^p(\cdot,E)$ is an \textbf{Int}-sheaf.

$\big(\mathcal{C}\cap W^{1,p}\big)(\cdot,M)$ is well-defined, since the weak derivatives of weakly differentiable functions coincide almost everywhere, cf.~\cite[Proposition 1.5]{ConvScha16}. Given $t_0<t_0'<t_1'<t_1\in\mathbb{R}$, $x\in\big(\mathcal{C}\cap W^{1,p}\big)([t_0,t_1'],M)$ and $x'\in\big(\mathcal{C}\cap W^{1,p}\big)([t_0',t_1],M)$ with $x\vert_{[t_0',t_1']} = x'\vert_{[t_0',t_1']}$, the function 
\begin{align*}
z: [t_0,t_1]\to E,\qquad t\mapsto\begin{cases}
x(t), & t\in[t_0,t_1']\\
x'(t), & t\in[t_0',t_1]
\end{cases}
\end{align*}
is the unique continuous function satisfying $z\vert_{[t_0,t_1']} = x$ and $z\vert_{[t_0',t_1]} = x'$. Let $X$ and $X'$ be weak derivatives of $x$ and $x'$, respectively. Then, both $X\vert_{[t_0',t_1']}$ and $X'\vert_{[t_0',t_1']}$ are weak derivatives of $z\vert_{[t_0',t_1']}$ and do therefore coincide almost everywhere. In particular,
\begin{align*}
Z:[t_0,t_1]\to \mathcal{T}M,\qquad t\mapsto\begin{cases}
X(t), & t\in[t_0,t_0'),\\
\frac{1}{2}\big(X(t)+X'(t)\big), & t\in [t_0',t_1'],\\
X'(t), & t\in(t_1',t_1]
\end{cases}
\end{align*}
is, by $Z = \chi_{[t_0,t_1']}X+\chi_{[t_1',t_1]}X'$ almost everywhere, in $\mathcal{L}^p_c([t_0,t_1],\mathcal{T}M)$, and $Z\vert_{[t_0,t_1']}$ and $Z\vert_{[t_0',t_1]}$ are weak derivatives of $x$ and $x'$, respectively. Hence, if $\varphi_U:U\to\mathbb{R}^n$ is a chart, then, for all $f\in\mathcal{C}_c^\infty\big(z^{-1}(U),\mathbb{R}^n\big)$, we have
\begin{align*}
\int_{z^{-1}(U)}\tfrac{\mathrm{d}}{\mathrm{d}t}f(\tau)^\top z^*\varphi_U(\tau)\,\mathrm{d}\tau & = \int_{z^{-1}(U)\cap[t_0,t_11']}\tfrac{\mathrm{d}}{\mathrm{d}t}f(\tau)^\top x^*\varphi_U(\tau)\,\mathrm{d}\tau\\
&\qquad +\int_{z^{-1}(U)\cap[t_1',t_0']}\tfrac{\mathrm{d}}{\mathrm{d}t}f(\tau)^\top (x')^*\varphi_U(\tau)\,\mathrm{d}\tau\\
& = -\int_{z^{-1}(U)\cap[t_0,t_11']} (\mathrm{d}\varphi_U
X(\tau))^\top f(\tau)\,\mathrm{d}\tau\\
&\qquad -\int_{z^{-1}(U)\cap[t_1',t_0']}(\mathrm{d}\varphi_U
X'(\tau))^\top f(\tau)^\top\,\mathrm{d}\tau\\
& = -\int_{z^{-1}(U)}(\mathrm{d}\varphi_U
Z(\tau))^\top f(\tau)^\top\,\mathrm{d}\tau
\end{align*}
and hence $Z$ is a week derivative of $z$ so that $z\in\big(\mathcal{C}\cap W^{1,p}\big)([t_0,t_1],M)$. This shows that $\big(\mathcal{C}\cap W^{1,p}\big)(\cdot,M)$ is an \textbf{Int}-sheaf.
\end{proof}

\section{Behaviours of port-Hamiltonian systems}

In this section, we describe the behaviours of port-Hamiltonian systems. For simplicity, we do not treat the variables of the resistive port as latent variables, which do not appear explicitly in the behaviour. In particular for port-Hamiltonian ODE systems~\eqref{eq:pH_ODE_system} this is not intuitive, since the variables of the resistive port are implicit and non-uniquely determined. However, we have defined port-Hamiltonian systems by their geometrical data. Therefore, keeping all port-variables as explicit part of the behaviour removes one main contributor to the potential ambiguity in the correspondence of port-Hamiltonian systems and their behaviours.

The most straightforward to define behaviour is the classical behaviour, where it is assumed that the port-variables are continuous and the state trajectory is (at least) once continuously differentiable. The corresponding behaviour is given as follows.

\begin{definition-proposition}
Let $(\pi:E_r\oplus E_p\to M,\mathcal{D},\mathcal{R},H)$ be the geometrical data of the port-Hamiltonian system
\begin{equation}\label{eq:pHS_classical}
\begin{aligned}
\left(\tfrac{\mathrm{d}}{\mathrm{d}t}x,f,-y,\mathrm{d}H_x,e,u\right) & \in\mathcal{D}_x\\
(f,e) & \in\mathcal{R}_x
\end{aligned}
\end{equation}
Define, for each $t_0<t_1\in\mathbb{R}$, $\mathfrak{B}^{\mathrm{c}}_{\mathcal{D},\mathcal{R},H}([t_0,t_1])$ as
\begin{align*}
\set{(f,y,e,u)\in\mathcal{C}([t_0,t_1],E_r\oplus E_p\oplus E_r^*\oplus E_p^*)\left\vert\!\begin{array}{l}
x = (f,y)^*\pi\in\mathcal{C}^1([t_0,t_1],M),\\
\eqref{eq:pHS_classical}~\mathrm{satisfied~pointwise}
\end{array}\right.\!\!}.
\end{align*}
$\mathfrak{B}^{\mathrm{c}}_{\mathcal{D},\mathcal{R},H}$ is an \textbf{Int}-sheaf of functions, the \textit{classical behaviour} of~\eqref{eq:pHS_classical}.
\end{definition-proposition}
\begin{proof}
Since the glueing (and restriction) in $\mathcal{C}^1(\cdot,M)$ is defined by the glueing (restriction) in $\mathcal{C}(\cdot,M)$ and since~\eqref{eq:pHS_classical} is supposed to be satisfied pointwise, the glueing (restriction) of elements of $\mathfrak{B}^{\mathrm{c}}_{\mathcal{D},\mathcal{R},H}$ in $\mathcal{C}(\cdot,E_r\oplus E_p\oplus E_r^*\oplus E_p^*)$ remains in $\mathfrak{B}^{\mathrm{c}}_{\mathcal{D},\mathcal{R},H}$. This verifies that $\mathfrak{B}^{\mathrm{c}}_{\mathcal{D},\mathcal{R},H}$ is indeed an \textbf{Int}-sheaf.
\end{proof}

The behaviour is not uniquely determined by the geometric data of the port-Hamiltonian system. In particular, there are degenerate cases in which no solution exists at all.

\begin{example}
There are port-Hamiltonian systems with empty behaviour. The simple example is the system
\begin{align*}
\mathrm{d}H_x = 0
\end{align*}
in case that $\mathrm{d}H$ vanishes nowhere (e.g. $H(x)=\exp{x}$). The underlying Dirac structure is the constant Dirac structure $D := \mathbb{R}\times\set{0}\leq \mathbb{R}\oplus\mathbb{R}$, and the external and resistive ports are zero-dimensional. In particular, the assignment of the geometric data of a port-Hamiltonian system to its (classical) behaviour is not one-to-one.
\end{example}

Next, we shall construct the weak behaviour, where we follow the lead of~\cite{Reis25b}, which considers weak solutions of port-Hamiltonian systems in trivial vector bundles.

\begin{definition}[cf.~\cite{Reis25b}]
Let $E\to M$ be a vector bundle, and consider a curve $(f,e)\in\mathcal{L}^1([t_0,t_1],E\oplus E^*)$ with continuous shadow $x:[t_0,t_1]\to M$. The space of test-sections for a subvector bundle $\mathcal{E}\leq E\oplus E^*$ along $x$ is
\begin{small}
\begin{align*}
\mathcal{C}_c([t_0,t_1],x^*\mathcal{E}) := \set{(\varphi,\varepsilon)\in\mathcal{C}([t_0,t_1],x^*\mathcal{E})\big\vert\,\mathrm{cl}\big([t_0,t_1]\setminus(\varphi,\varepsilon)^{-1}(\set{0_\mathcal{D}})\big)\subseteq(t_0,t_1)},
\end{align*}
\end{small}
where $x^*\mathcal{E}$ denotes the pullback bundle in the topological manifolds.
\begin{enumerate}[(i)]
\item Let $\mathcal{D}\leq E\oplus E^*$ be a Dirac structure. $(f,e)$ satisfies $(f,e)\in\mathcal{D}$ \textit{in the weak sense} if, and only if,
\begin{align*}
\forall(\varphi,\varepsilon)\in\mathcal{C}_c([t_0,t_1],x^*\mathcal{D}): \int_{[t_0,t_1]} \langle \varepsilon_{x(t)},f(t)\rangle+\langle e(t),\varphi_{x(t)}\rangle\,\mathrm{d}t = 0.
\end{align*}
\item Let $\mathcal{L}\leq E\oplus E^*$ be a Lagrangian subbundle. $(f,e)$ satisfies $(f,e)\in\mathcal{L}$ \textit{in the weak sense} if, and only if,
\begin{align*}
\forall(\varphi,\varepsilon)\in\mathcal{C}_c([t_0,t_1],x^*\mathcal{L}): \int_{[t_0,t_1]} \langle \varepsilon_{x(t)},f(t)\rangle-\langle e(t),\varphi_{x(t)}\rangle\,\mathrm{d}t = 0.
\end{align*}
\end{enumerate}
\end{definition}

\begin{remark}
The collection of spaces of test sections is not an \textbf{Int}-sheaf of functions. Since each test section can be extended by zero to a test-section on a larger interval, it has the natural structure of an \textbf{Int}-\textit{cosheaf}. This structure is not exploited in this note, but can (in the trivial case) be used to define distributional solutions.
\end{remark}

To construct a behaviour, we utilise that the weak solutions of the algebraic system given by a Dirac structure or Lagrangian subbundle have the natural structure of an \textbf{Int}-sheaf.

\begin{lemma}
Let $\pi:E\to M$ be a vector bundle, and let $\mathcal{D}\leq E\oplus E^*$ be a Dirac structure. Define, for each $t_0<t_1\in\mathbb{R}$,
\begin{align*}
\mathfrak{W}([t_0,t_1],\mathcal{D}) := \set{(f,e)\in\mathcal{L}^1_c([t_0,t_1],E\oplus E^*)\,\big\vert\,
(f,e)\in\mathcal{D}~\mathrm{in~the~weak~sense}}
\end{align*}
and consider the restriction of functions as restriction morphisms. $\mathfrak{W}(\cdot,\mathcal{D})$ is an \textbf{Int}-sheaf. The same holds true if $\mathcal{D}$ is a Lagrangian subbundle.
\end{lemma}
\begin{proof}
By Lemma~\ref{lem:weak_sheaves}, $\mathcal{L}^1_c(\cdot,E\oplus E^*)$ is an \textbf{Int}-sheaf. Let $t_0\leq t_0'<t_1'\leq t_1$, and consider $(f,e)\in\mathfrak{W}([t_0,t_1'],\mathcal{D})$ and $(f',e')\in\mathfrak{W}([t_0',t_1],\mathcal{D})$ so that $(f,e)\vert_{[t_0',t_1']} = (f',e')\vert_{[t_0',t_1']}$. The function
\begin{align*}
(\overline{f},\overline{g}):[t_0,t_1]\to E\oplus E^*,\qquad t\mapsto \begin{cases}
(f,e)(t), & t\in [t_0,t_1'],\\
(f',e')(t), & t\in [t_0',t_1]
\end{cases}
\end{align*}
is the (unique) glueing of $(f,e)$ and $(f',e')$ in $\mathcal{L}^1_c(\cdot,E\oplus E^*)$ and it remains to verify that $(\overline{f},\overline{g})\in \mathfrak{W}([t_0,t_1],\mathcal{D})$. Let $(\varphi,\varepsilon)\in\mathcal{C}_c^\infty([t_0,t_1],\mathcal{D})$ and consider a smooth function $\lambda_0\in\mathcal{C}^\infty(\mathbb{R})$ with
\begin{align*}
\lambda_0\vert_{\big[t_1'-\frac{1}{3}(t_1'-t_0'),\infty\big)}\equiv 1~\mathrm{and}~\lambda_0\vert_{\big(-\infty,t_0'+\frac{1}{3}(t_1'-t_0')\big]}\equiv 0.
\end{align*}
Put $\lambda := 1-\lambda_0\vert_{[t_0,t_1]}$ and $\lambda':=\lambda_0\vert_{[t_0,t_1]}$. Then, $\big(\lambda(\varphi,\varepsilon)\big)\vert_{[t_0,t_1']}\in\mathcal{C}_c^\infty([t_0,t_1'],\mathcal{D})$ and $\big(\lambda'(\varphi,\varepsilon)\big)\vert_{[t_0',t_1]}\in\mathcal{C}_c^\infty([t_0',t_1],\mathcal{D})$ so that
\begin{align*}
& \int_{[t_0,t_1]} \langle \varepsilon_{x(t)},\overline{f}(t)\rangle+\langle \overline{e}(t),\varphi_{x(t)}\rangle\,\mathrm{d}t\\
&\qquad = \int_{[t_0,t_1]} \langle (\lambda(t)+\lambda'(t))\varepsilon_{x(t)},\overline{f}(t)\rangle+\langle \overline{e}(t),(\lambda(t)+\lambda'(t))\varphi_{x(t)}\rangle\,\mathrm{d}t\\
&\qquad = \int_{[t_0,t_1']} \langle \lambda(t)\varepsilon_{x(t)},{f}(t)\rangle+\langle {e}(t),\lambda(t)\varphi_{x(t)}\rangle\,\mathrm{d}t\\
&\qquad\qquad + \int_{[t_0',t_1]} \langle \lambda'(t)\varepsilon_{x(t)},{f'}(t)\rangle+\langle {e'}(t),\lambda'(t)\varphi_{x(t)}\rangle\,\mathrm{d}t\\
&\qquad = 0.
\end{align*}
This verifies that $(\overline{f},\overline{g})\in \mathfrak{W}([t_0,t_1],\mathcal{D})$. For the case of Lagrangian subbundles, the pseudo-Euclidean metric in the integral is replaced with the canonical symplectic form.
\end{proof}

This gives the weak behaviour as the weak analogue of the classical behaviour, defined as follows.

\begin{definition-proposition}
Let $(\pi:E_r\oplus E_p\to M,\mathcal{D},\mathcal{R},H)$ be the geometrical data of the port-Hamiltonian system
\begin{equation}\label{eq:pHS_weak}
\begin{aligned}
\left(\tfrac{\mathrm{d}}{\mathrm{d}t}x,f,-y,\mathrm{d}H_x,e,u\right) & \in\mathcal{D}_x\\
(f,e) & \in\mathcal{R}_x
\end{aligned}
\end{equation}
Define, for each $t_0<t_1\in\mathbb{R}$, $\mathfrak{B}^{\mathrm{w}}_{\mathcal{D},\mathcal{R},H}([t_0,t_1])\subseteq\mathcal{L}^1([t_0,t_1],E_r\oplus E_p\oplus E_r^*\oplus E_p^*)$ as
\begin{align*}
\set{(f,y,e,u)\left\vert\!\begin{array}{l}
x = (f,y)^*\pi\in\big(\mathcal{C}\cap\mathcal{W}^{1,1}\big)([t_0,t_1],M),\\
\eqref{eq:pHS_classical}~\mathrm{satisfied~in~the~weak~sense}
\end{array}\right.\!\!}
\end{align*}
and consider the restriction of functions as restriction morphisms. $\mathfrak{B}^{\mathrm{w}}_{\mathcal{D},\mathcal{R},H}$ is an \textbf{Int}-sheaf, the \textit{weak behaviour} of~\eqref{eq:pHS_weak}.
\end{definition-proposition}
\begin{proof}
Since the glueing (and restriction) in $\mathcal{C}^1(\cdot,M)$ is defined by the glueing (restriction) in $\mathcal{C}(\cdot,M)$ and since~\eqref{eq:pHS_classical} is supposed to be satisfied pointwise, the glueing (restriction) of elements of $\mathfrak{B}^{\mathrm{c}}_{\mathcal{D},\mathcal{R},H}$ in $\mathcal{C}(\cdot,E_r\oplus E_p\oplus E_r^*\oplus E_p^*)$ remains in $\mathfrak{B}^{\mathrm{c}}_{\mathcal{D},\mathcal{R},H}$. This verifies that $\mathfrak{B}^{\mathrm{c}}_{\mathcal{D},\mathcal{R},H}$ is indeed an \textbf{Int}-sheaf.
\end{proof}

\section{Interconnection of port-Hamiltonian systems}

In the previous section, we clarified the definitions of classical and weak solutions of port-Hamiltonian systems. The natural structural properties of the collection of all solutions is encoded in the language of \textbf{Int}-sheaves. In this section, we may finally study the interconnection of port-Hamiltonian systems. First, we shall recall the geometric foundation, which facilitates the claim that interconnection of port-Hamiltonian system is structure preserving. 

\subsection{Composition of Dirac structures}\label{geometric_approach}

The composition of Dirac structures is the threefold composition of Dirac structures as relations. It is well-known that this produces, in the constant case, again a Dirac structure, cf.~\cite{Interconnection07,Interconnection18}. It is noteworthy that~\cite{Interconnection07} considers \textit{per se} a particular interconnection only, while~\cite{Interconnection18} presents the general statement. A quick calculation reveals, however, that both results are, in fact, equivalent, and that it suffices to consider the interconnection as a feedback relation (by a Dirac structure) over the external port-variables of a single system. To obtain from that result the classical interconnection of two (or finitely many) port-Hamiltonian systems, the single system is the composition. We first recall the statement in the constant case and provide a sketch of a proof, which is based on the proof of~\cite[Lemma 5.2]{Burs13}.

\begin{proposition}\label{prop:geometric_composition}
Let $V,W$ be real, finite-dimensional vector spaces, and let ${D}_1\leq (V\oplus W)\oplus (V^*\oplus W^*)$ and ${D}_2\leq W\oplus W^*$ be (constant) Dirac structures.
\begin{align*}
{D}_2\circ{D}_2 := \set{(v,\varphi)\in V\oplus V^*\,\big\vert\,\exists (w,\psi)\in{D}_2: (v,w,\varphi,\psi)\in D_1}
\end{align*}
is a Dirac structure.
\end{proposition}
\begin{proof}
Identify $D_2$ with $\set{(0,w,0,\psi)\in(V\oplus W)\oplus (V^*\oplus W^*)\,\big\vert\,(w,\psi)\in D_2}$. The annihilator with respect to the canonical pseudo-Euclidean metric is $D_2^\bbot = \set{(v,w,\varphi,\psi)\in(V\oplus W)\oplus (V^*\oplus W^*)\,\big\vert\,(w,\psi)\in D'}$. A direct calculation gives that $((D_1\cap D_2^\bbot)+D_2^\bbot = (D_1\cap D_2^\bbot)+D_2$ is a Dirac structure, which coincides with the direct sum
\begin{align*}
{D}_1\circ{D}_2\boxplus D_2 := \set{(v,w,\varphi,\psi)\,\big\vert\,(v,\varphi)\in{D}_2\circ{D}_2, (w,\psi)\in D_1}.
\end{align*}
This is the case if, and only if, both ${D}_1\circ{D}_2$ and $D_2$ are Dirac structures.
\end{proof}

The consequence of this result is the classical composition of Dirac structures. In particular, since non-constant Dirac structures are fibrewise constant Dirac structures, the following result is obtained.

\begin{corollary}
Let $E_1\to M$, $E_2\to M$ and $F_1\to N$, $F_2\to N$ be vector bundles, and let $\mathcal{D}_1\leq (E_1\oplus E_2)\oplus(E_1^*\oplus E_2^*)$ and $\mathcal{D}_2\leq (F_1\oplus F_2)\oplus(F_1^*\oplus F_2^*)$ be Dirac structures. Given a third Dirac structure $\mathcal{D}_{\mathrm{int}}\leq (E_2\times F_2)\oplus(E_2^*\times F_2^*)$, the composition
\begin{align*}
\mathcal{D}_1\Vert_{\mathcal{D}_{\mathrm{int}}}\mathcal{D}_2 := \set{(e_1,f_1,\varepsilon_1,\varphi_1)\,\left\vert\,\exists (e_2,f_2,\varepsilon_2,\varphi_2)\in\mathcal{D}_{\mathrm{int}}:\begin{array}{l}
(e_1,e_2,\varepsilon_1,\varepsilon_2)\in\mathcal{D}_1,\\
(f_1,f_2,\varphi_1,\varphi_2)\in\mathcal{D}_2
\end{array}\right.\!\!}
\end{align*}
is a regular, but not necessarily smooth, generalised subbundle, which is fibrewise a constant Dirac structure.
\end{corollary}
\begin{proof}
For each $(x,y)\in M\times N$,
\begin{align*}
\big(\mathcal{D}_1\Vert_{\mathcal{D}_{\mathrm{int}}}\mathcal{D}_2\big)_{(x,y)} = \big((\mathcal{D}_1)_x\boxplus(\mathcal{D}_2)_y\big)\circ(\mathcal{D}_{\mathrm{int}})_{(x,y)}.
\end{align*}
\end{proof}

Combining the results of~\cite{Burs13} on sufficient conditions of smoothness of forward and backward of Dirac structures with the insight of~\cite{Interconnection18} on the intimate connection between those and composition of Dirac structures gives a sufficient condition on smoothness. This sufficient condition is the following well-known \textit{clean intersection} condition, cf.~\cite[Proposition 3.15]{Vyso20}.

\begin{proposition}
Let $E\to M$ and $F\to M$ be vector bundles over the same manifold, and let $\mathcal{D}_1\leq (E\oplus F)\oplus(E^*\oplus F^*)$ and $\mathcal{D}_2\leq F\oplus F^*$ be Dirac structures. If
\begin{align}\label{eq:intersection}
\mathcal{D}_1\cap\mathcal{D}_2 = \set{(0_E(x),v,0_{E^*}(x),w)\in\mathcal{D}_1\,\big\vert\,(v,w)\in\mathcal{D}_2}
\end{align}
has constant rank, then $\mathcal{D}_1\circ\mathcal{D}_2$ is a Dirac structure.
\end{proposition}
\begin{proof}
Recall that the intersection of subvector bundles is a subvector bundle if, and only if, it has constant rank~\cite[Theorem 10.34]{Lee13}, and observe further
\begin{align*}
\forall x\in M: \big(\mathcal{D}_1\circ\mathcal{D}_2\big)_x \simeq \big(\big(\mathcal{D}_1\cap\mathcal{D}_2^\bbot\big)/\big(\mathcal{D}_1\cap\mathcal{D}_2\big)\big)_x,
\end{align*}
where $\mathcal{D}_1\cap\mathcal{D}_2^{\bbot} = \set{(f,v,e,w)\in\mathcal{D}_1\,\big\vert\,(v,w)\in\mathcal{D}_2}$. Since $\mathcal{D}_1\cap\mathcal{D}_2$ has constant rank and $\mathcal{D}_1\circ\mathcal{D}_2$ has constant rank, $\mathcal{D}_1\cap\mathcal{D}_2^\bbot$ has constant rank. Therefore, both $\mathcal{D}_1\cap\mathcal{D}_2$ and $\mathcal{D}_1\cap\mathcal{D}_2^\bbot$ are subvector bundles. Since the isomorphism $\mathcal{D}_1\circ\mathcal{D}_2 \simeq \big(\mathcal{D}_1\cap\mathcal{D}_2^\bbot\big)/\big(\mathcal{D}_1\cap\mathcal{D}_2\big)$ is smooth, this implies that $\mathcal{D}_1\circ\mathcal{D}_2$ is smooth and thus indeed a Dirac structure.
\end{proof}

The clean intersection condition is sufficient, but not necessary. We demonstrate this with an example.

\begin{example}\label{ex:sufficient_not_necessary}
Let $\lambda\in\mathcal{C}^\infty(\mathbb{R})$ and consider the smooth antisymmetric matrix field
\begin{align*}
J_{\lambda}:\mathbb{R}\to\mathbb{R}^{3\times 3},\qquad t\mapsto\begin{bmatrix}
0 & \lambda(t) & \lambda(t)^2\\
-\lambda(t) & 0 & 0\\
-\lambda(t)^2 & 0 & 0
\end{bmatrix}.
\end{align*}
The induced Dirac structure in the trivial bundle $p_1:\mathbb{R}\times\mathbb{R}^3\to\mathbb{R}$ is
\begin{align*}
\mathcal{D}_{1} := \set{(t,X,J_{\lambda}(t)X)\,\big\vert\,X\in\mathbb{R}^3,t\in\mathbb{R}}.
\end{align*}
Consider furthermore the Dirac structure
\begin{align*}
\mathcal{D}_{2} := \set{\left.\left(t,\begin{pmatrix}
x\\y
\end{pmatrix},\begin{pmatrix}
\lambda(t)^2y\\-\lambda(t)^2x
\end{pmatrix}\right)\,\right\vert\,t,x,y\in\mathbb{R}}.
\end{align*}
We show that $\mathcal{D}_{1}\circ\mathcal{D}_{2}$ is a Dirac structure. Let $t\in\mathbb{R}$. Then, we have $(t,a,b)\in\mathcal{D}_{1}\circ\mathcal{D}_{2}$ if, and only if, there exists some $x,y\in\mathbb{R}$ with 
\begin{align*}
\lambda(t)^2y & = -\lambda(t)a\\
\lambda(t)^2x & = \lambda(t)^2a\\
b & = \lambda(t)x+\lambda(t)^2y.
\end{align*}
Plugging the first two equations into the third gives
\begin{align*}
b & = \lambda(t)x+\lambda(t)^2y = \lambda(t)x-\lambda(t)a = \left.\begin{cases}
0, & \lambda(t) = 0,\\
\lambda(t)a-\lambda(t)a, & \lambda(t)\neq 0
\end{cases}\right\rbrace = 0.
\end{align*}
Therefore, $\mathcal{D}_{1}\circ\mathcal{D}_{2} = \mathbb{R}\times(\mathbb{R}\oplus\set{0})$ is indeed a Dirac structure. The intersection $\big(\mathcal{D}_1\cap\mathcal{D}_2\big)_t$ can be calculated by considering $a = b = 0$ so that
\begin{align*}
\forall t\in\mathbb{R}: \big(\mathcal{D}_1\cap\mathcal{D}_2\big)_t = \begin{cases}
\set{(0,x,y,0,0,0)\,\big\vert\,x,y\in\mathbb{R}}, & \lambda(t) = 0,\\
\set{0}, & \lambda(t) \neq 0.
\end{cases}
\end{align*}
In particular, in case that $\lambda$ vanishes neither nowhere nor everywhere, the clean intersection condition is not satisfied, but $\mathcal{D}_{1}\circ\mathcal{D}_{2}$ is smooth.

Replacing $\lambda^2$ with another smooth function $\mu$ which has a zero that is not a zero of $\lambda$ and does not vanishes everywhere gives an example of Dirac structures whose composition is not smooth:
\begin{align*}
\forall t\in\mathbb{R}: \big(\mathcal{D}_{1}\circ\mathcal{D}_{2}\big)_t = \begin{cases}
\set{0}\oplus\mathbb{R}, & \mu(t) = 0, \lambda(t)\neq 0,\\
\mathbb{R}\oplus\set{0}, & \mathrm{else}.
\end{cases}
\end{align*}
\end{example}

We define an \textit{admissible interconnection constraint} for port-Hamiltonian systems by requiring that the composition of the Dirac structures is smooth. The clean intersection condition is a sufficient, but not necessary, condition for admissibility. This gives the formal definition of the geometric interconnection of port-Hamiltonian systems as follows.

\begin{definition}
Let $E_s^j\to M_j$, $E_r^j\to M_j$, $E_{p_i}^j\to M_j$, $E_{p_e}^j\to M_j$ be smooth vector bundles, and let $\mathcal{D}_j\leq \big(E_s^j\oplus E_r^j\oplus E_{p_i}^j\oplus E_{p_e}^j\big)\oplus\big((E_s^j)^*\oplus (E_r^j)^*\oplus (E_{p_i}^j)^*\oplus (E_{p_e}^j)^*\big)$ be Dirac structures, $j\in\set{1,2}$. 
\begin{enumerate}[(i)]
\item A Dirac structure $\mathcal{D}_{\mathrm{int}}\leq \big(E_{p_e}^1\oplus E_{p_e}^2\big) \big((E_{p_e}^1)^*\oplus (E_{p_e}^2)^*\big)$ is an \textit{admissible interconnection constraint} for the pair $(\mathcal{D}_1,\mathcal{D}_2)$ if, and only if, $\mathcal{D}_1\Vert_{\mathcal{D}_{\mathrm{int}}}\mathcal{D}_2$ is smooth.
\item If $\mathcal{D}_{\mathrm{int}}$ is an admissible interconnection constraint, the \textit{geometric interconnection} of the port-Hamiltonian systems
\begin{align*}
\begin{pmatrix}
\tfrac{\mathrm{d}}{\mathrm{d}t}x_j, f_j, -y_j^1, -y_j^2, \mathrm{d}H_j, e_j, u_j^1, u_j^2
\end{pmatrix} \in\mathcal{D}_j,\quad \begin{pmatrix}
f_j,e_j
\end{pmatrix}\in\mathcal{R}_j, \qquad j\in\set{1,2}
\end{align*}
with respect to the interconnection constraint
\begin{align*}
\begin{pmatrix}
-y_1^2, -y_2^2, u_1^2, u_2^2
\end{pmatrix}\in\mathcal{D}_{\mathrm{int}}
\end{align*}
is the port-Hamiltonian system
\begin{equation*}
\begin{aligned}
\begin{pmatrix}
\tfrac{\mathrm{d}}{\mathrm{d}t}x_1,\tfrac{\mathrm{d}}{\mathrm{d}t}x_2, f_1, f_2, -y_1^1, -y_2^1, \mathrm{d}H_1, \mathrm{d}H_2, e_1, e_2, u_1^1, u_2^1
\end{pmatrix} & \in\mathcal{D}_1\Vert_{\mathcal{D}_{\mathrm{int}}}\mathcal{D}_2,\\
\begin{pmatrix}
f_1,f_2,e_1,e_2
\end{pmatrix} & \in\mathcal{R}_1\boxplus\mathcal{R}_2.
\end{aligned}
\end{equation*}
\end{enumerate}
\end{definition}

\subsection{Behavioural approach}\label{sec:behavioural_approach}

In the geometric interconnection, the definition of the composition of Dirac structures guarantees that the equations are pointwise identical. The external port-variables on which the interconnection constraints are imposed, however, only remain implicit and pointwise. In particular, for each element of the interconnected system, the existence of a general curve realising the interconnection constraint is guaranteed, but no statement on the regularity can be made \textit{a priori}. From a dynamical systems' point of view, we would like to see that the systems are equivalent so that we want to impose regularity assumptions on \textit{some} curve realising the interconnection constraint. This is formalised in the \textit{behavioural interconnection}, which is given for both the classical and the weak behaviour.

\begin{definition-proposition}\label{def-prop:behavioural_interconnection}
Let 
\begin{itemize}
\item $E_s^j\to M_j$, $E_r^j\to M_j$, $E_{p_i}^j\to M_j$, $E_{p_e}^j\to M_j$ be smooth vector bundles,
\item $\mathcal{D}_j\leq \big(E_s^j\oplus E_r^j\oplus E_{p_i}^j\oplus E_{p_e}^j\big)\oplus\big((E_s^j)^*\oplus (E_r^j)^*\oplus (E_{p_i}^j)^*\oplus (E_{p_e}^j)^*\big)$ be Dirac structures,
\item $\mathcal{R}_j\leq E_r^j\oplus (E_r^j)^*$ be non-negative Lagrangian subbundles,
\item $H_j\in\mathcal{C}^\infty(M_j)$ be Hamiltonian functions,
\item $\mathcal{D}_{\mathrm{int}}\leq \big(E_{p_e}^1\oplus E_{p_e}^2\big) \big((E_{p_e}^1)^*\oplus (E_{p_e}^2)^*\big)$ be an {admissible interconnection constraint} for $(\mathcal{D}_1,\mathcal{D}_2)$.
\end{itemize}
For $\mathfrak{t}\in\set{\mathrm{c},\mathrm{w}}$, the \textit{behavioural interconnection} $\big(\mathfrak{B}^{\mathfrak{t}}_{\mathcal{D}_1,\mathcal{R}_1,H_1}\Vert_{\mathcal{D}_{\mathrm{int}}}\mathfrak{B}^{\mathfrak{t}}_{\mathcal{D}_2,\mathcal{R}_2,H_2}\big)([t_0,t_1])$ is, for each $t_0<t_1\in\mathbb{R}$, defined as
\begin{align*}
\set{(f_1,f_2,y_1^1,y_2^1,e_1,e_2,u_1^1,u_2^1):[t_0,t_1]\to\mathcal{E}\,\left\vert\,\begin{array}{l}
\exists (y_1^2,y_2^2,u_1^2,u_2^2):[t_0,t_1]\to\mathcal{D}_{\mathrm{int}}:\\
\quad (f_1,y_1^1,-y_1^2,e_1,u_1^1,u_1^2)\\
\hspace{2cm}\in \mathfrak{B}^{\mathfrak{t}}_{\mathcal{D}_1,\mathcal{R}_1,H_1}([t_0,t_1]),\\
\quad (f_2,y_2^1,-y_2^2,e_2,u_2^1,u_2^2)\\
\hspace{2cm}\in \mathfrak{B}^{\mathfrak{t}}_{\mathcal{D}_2,\mathcal{R}_2,H_2}([t_0,t_1])
\end{array}\right.}
\end{align*}
with
\begin{align*}
\mathcal{E} := \big(E_s^1\oplus E_s^2\oplus E_r^1\oplus E_r^2\oplus E_{p_i}^1\oplus E_{p_i}^2\big)\oplus\big(E_s^1\oplus E_s^2\oplus E_r^1\oplus E_r^2\oplus E_{p_i}^1\oplus E_{p_i}^2\big)^*.
\end{align*}
$\big(\mathfrak{B}^{\mathfrak{t}}_{\mathcal{D}_1,\mathcal{R}_1,H_1}\Vert_{\mathcal{D}_{\mathrm{int}}}\mathfrak{B}^{\mathfrak{t}}_{\mathcal{D}_2,\mathcal{R}_2,H_2}\big)$ is an \textbf{Int}-sheaf of functions.
\end{definition-proposition}
\begin{proof}
Since both $\mathfrak{B}_{\mathfrak{t},\mathcal{D}_j,\mathcal{R}_j}$, $j\in\set{1,2}$, are sheaves, $\big(\mathfrak{B}^{\mathfrak{t}}_{\mathcal{D}_1,\mathcal{R}_1,H_1}\Vert_{\mathcal{D}_{\mathrm{int}}}\mathfrak{B}^{\mathfrak{t}}_{\mathcal{D}_2,\mathcal{R}_2,H_2}\big)$ is a separable presheaf. Let $t_0<\vartheta_0<\vartheta_1<t_1\in\mathbb{R}$, and let 
\begin{align*}
(f_1,f_2,y_1^1,y_2^1,e_1,e_2,u_1^1,u_2^1) & \in\big(\mathfrak{B}^{\mathfrak{t}}_{\mathcal{D}_1,\mathcal{R}_1,H_1}\Vert_{\mathcal{D}_{\mathrm{int}}}\mathfrak{B}^{\mathfrak{t}}_{\mathcal{D}_2,\mathcal{R}_2,H_2}\big)([t_0,\vartheta_1]),\\
(\widehat{f}_1,\widehat{f}_2,\widehat{y}_1^1,\widehat{y}_2^1,\widehat{e}_1,\widehat{e}_2,\widehat{u}_1^1,\widehat{u}_2^1) & \in\big(\mathfrak{B}^{\mathfrak{t}}_{\mathcal{D}_1,\mathcal{R}_1,H_1}\Vert_{\mathcal{D}_{\mathrm{int}}}\mathfrak{B}^{\mathfrak{t}}_{\mathcal{D}_2,\mathcal{R}_2,H_2}\big)([\vartheta_0,t_1])
\end{align*}
satisfying
\begin{align*}
(f_1,f_2,y_1^1,y_2^1,e_1,e_2,u_1^1,u_2^1)\vert_{[\vartheta_0,\vartheta_1]} = (\widehat{f}_1,\widehat{f}_2,\widehat{y}_1^1,\widehat{y}_2^1,\widehat{e}_1,\widehat{e}_2,\widehat{u}_1^1,\widehat{u}_2^1)\vert_{[\vartheta_0,\vartheta_1]}
\end{align*}
with $(y_1^2,y_2^2,u_1^2,u_2^2):[t_0,\vartheta_1]\to\mathcal{D}_{\mathrm{int}}$ and $(\widehat{y}_1^2,\widehat{y}_2^2,\widehat{u}_1^2,\widehat{u}_2^2):[\vartheta_0,t_1]\to\mathcal{D}_{\mathrm{int}}$ realising the defining characterisation of the behavioural interconnection. As functions, there exists a unique function
\begin{align*}
(\overline{f}_1,\overline{f}_2,\overline{y}_1^1,\overline{y}_2^1,\overline{e}_1,\overline{e}_2,\overline{u}_1^1,\overline{u}_2^1):[t_0,t_1]\to\mathcal{E}
\end{align*}
so that
\begin{align*}
(f_1,f_2,y_1^1,y_2^1,e_1,e_2,u_1^1,u_2^1) = (\overline{f}_1,\overline{f}_2,\overline{y}_1^1,\overline{y}_2^1,\overline{e}_1,\overline{e}_2,\overline{u}_1^1,\overline{u}_2^1)\vert_{[t_0,\vartheta_1]},\\
(f_1,f_2,y_1^1,y_2^1,e_1,e_2,u_1^1,u_2^1) = (\overline{f}_1,\overline{f}_2,\overline{y}_1^1,\overline{y}_2^1,\overline{e}_1,\overline{e}_2,\overline{u}_1^1,\overline{u}_2^1)\vert_{[\vartheta_0,t_1]}
\end{align*}
Choose a function $\varepsilon\in\mathcal{C}^\infty([t_0,t_1])$ with $\ran\varepsilon\subseteq[0,1]$ and
\begin{align*}
\varepsilon\vert_{\left[t_0,\frac{2}{3}\vartheta_0+\frac{1}{3}\vartheta_1\right]}\equiv 0\quad\mathrm{and}\quad\varepsilon\vert_{\left[\frac{2}{3}\vartheta_1+\frac{1}{3}\vartheta_0,t_1\right]}\equiv 1.
\end{align*}
For ease of readability, denote the extensions of $\varepsilon\vert_{[\vartheta_0,t_1]}(\widehat{y}_1^2,\widehat{y}_2^2,\widehat{u}_1^2,\widehat{u}_2^2)$ by zero to $[t_0,t_1]$ by $\varepsilon(\widehat{y}_1^2,\widehat{y}_2^2,\widehat{u}_1^2,\widehat{u}_2^2)$, and define likewise $(1-\varepsilon)(y_1^2,y_2^2,u_1^2,u_2^2)$ and $\varepsilon(\widehat{f}_1,\widehat{f}_2,\widehat{y}_1^1,\widehat{y}_2^1,\widehat{e}_1,\widehat{e}_2,\widehat{u}_1^1,\widehat{u}_2^1)$ and $(1-\varepsilon)(f_1,f_2,y_1^1,y_2^1,e_1,e_2,u_1^1,u_2^1)$. Both $\varepsilon(\widehat{y}_1^2,\widehat{y}_2^2,\widehat{u}_1^2,\widehat{u}_2^2)$ and $(1-\varepsilon)(y_1^2,y_2^2,u_1^2,u_2^2)$ are, in case $\mathfrak{t} = \mathrm{c}$, continuous, and integrable when $\mathfrak{t} = \mathrm{w}$. Moreover,
\begin{align*}
\big(\varepsilon(\widehat{y}_1^2,\widehat{y}_2^2,\widehat{u}_1^2,\widehat{u}_2^2)+(1-\varepsilon)(y_1^2,y_2^2,u_1^2,u_2^2)\big)\vert_{[\vartheta_1,t_1]} & = (\widehat{y}_1^2,\widehat{y}_2^2,\widehat{u}_1^2,\widehat{u}_2^2)\vert_{[\vartheta_1,t_1]}\\
\big(\varepsilon(\widehat{y}_1^2,\widehat{y}_2^2,\widehat{u}_1^2,\widehat{u}_2^2)+(1-\varepsilon)(y_1^2,y_2^2,u_1^2,u_2^2)\big)\vert_{[t_0,\vartheta_0]} & = (y_1^2,y_2^2,u_1^2,u_2^2)\vert_{[t_0,\vartheta_0]}.
\end{align*}
Put $x_1:=\overline{f}_1^*\pi_1$ and $x_2:=\overline{f}_2^*\pi_2$ and define
\begin{align*}
(\overline{y}_1^2,\overline{y}_2^2,\overline{u}_1^2,\overline{u}_2^2) := \varepsilon(\widehat{y}_1^2,\widehat{y}_2^2,\widehat{u}_1^2,\widehat{u}_2^2)+(1-\varepsilon)(y_1^2,y_2^2,u_1^2,u_2^2): [t_0,t_1]\to\mathcal{D}_{\mathrm{int}}.
\end{align*}
Then, we obtain
\begin{align*}
& \begin{pmatrix}
\tfrac{\mathrm{d}}{\mathrm{d}t}x_1,\overline{f}_1,-\overline{y}_1^1,\overline{y}_1^2,\mathrm{d}H^1_{x_1},\overline{e}_1,\overline{u}_1^1,\overline{u}_1^2
\end{pmatrix}\\
& \hspace{3cm} = \varepsilon\begin{pmatrix}
\tfrac{\mathrm{d}}{\mathrm{d}t}x_1,\overline{f}_1,-\overline{y}_1^1,\widehat{y}_1^2,\mathrm{d}H^1_{x_1},\overline{e}_1,\overline{u}_1^1,\widehat{u}_1^2
\end{pmatrix}\\
& \hspace{3.5cm} + (1-\varepsilon)\begin{pmatrix}
\tfrac{\mathrm{d}}{\mathrm{d}t}x_1,\overline{f}_1,-\overline{y}_1^1,{y}_1^2,\mathrm{d}H^1_{x_1},\overline{e}_1,\overline{u}_1^1,{u}_1^2
\end{pmatrix}\\
& \hspace{3cm} = \varepsilon\begin{pmatrix}
\tfrac{\mathrm{d}}{\mathrm{d}t}x_1,\widehat{f}_1,-\widehat{y}_1^1,\widehat{y}_1^2,\mathrm{d}H^1_{x_1},\widehat{e}_1,\widehat{u}_1^1,\widehat{u}_1^2
\end{pmatrix}\\
& \hspace{3.5cm} + (1-\varepsilon)\begin{pmatrix}
\tfrac{\mathrm{d}}{\mathrm{d}t}x_1,{f}_1,-{y}_1^1,{y}_1^2,\mathrm{d}H^1_{x_1},{e}_1,{u}_1^1,{u}_1^2
\end{pmatrix}\in\left(\mathcal{D}_1\right)_{x_1}
\end{align*}
and analogously
\begin{align*}
\begin{pmatrix}
\tfrac{\mathrm{d}}{\mathrm{d}t}x_1,\overline{f}_2,-\overline{y}_2^1,\overline{y}_2^2,\mathrm{d}H^2_{x_2},\overline{e}_2,\overline{u}_2^1,\overline{u}_2^2
\end{pmatrix} \in\left(\mathcal{D}_2\right)_{x_2}.
\end{align*}
Since $(\overline{f}_1,\overline{y}_1^1,\overline{e}_1,\overline{u}_1^1)$ and $(\overline{f}_2,\overline{y}_2^1,\overline{e}_2,\overline{u}_2^1)$ are, in case $\mathfrak{t} = \mathrm{c}$, continuous with continuously differentiable shadow, and integrable with continuous Sobolev shadow in case $\mathfrak{t} = \mathrm{w}$, we conclude therefore
\begin{align*}
\begin{pmatrix}
\overline{f}_1,\overline{y}_1^1,-\overline{y}_1^2,\overline{e}_1,\overline{u}_1^1,\overline{u}_1^2
\end{pmatrix} & \in \mathfrak{B}^{\mathfrak{t}}_{\mathcal{D}_1,\mathcal{R}_1,H_1}([t_0,t_1]),\\
\begin{pmatrix}
\overline{f}_2,\overline{y}_2^1,-\overline{y}_2^2,\overline{e}_2,\overline{u}_2^1,\overline{u}_2^2
\end{pmatrix} & \in \mathfrak{B}^{\mathfrak{t}}_{\mathcal{D}_1,\mathcal{R}_1,H_1}([t_0,t_1]),
\end{align*}
and hence
\begin{align*}
(\overline{f}_1,\overline{f}_2,\overline{y}_1^1,\overline{y}_2^1,\overline{e}_1,\overline{e}_2,\overline{u}_1^1,\overline{u}_2^1)\in \big(\mathfrak{B}^{\mathfrak{t}}_{\mathcal{D}_1,\mathcal{R}_1,H_1}\Vert_{\mathcal{D}_{\mathrm{int}}}\mathfrak{B}^{\mathfrak{t}}_{\mathcal{D}_2,\mathcal{R}_2,H_2}\big)([t_0,t_1]).
\end{align*}
\end{proof}

By definition of the interconnection of Dirac structures, the every element of the behavioural interconnection of port-Hamiltonian systems is an element of the behaviour of the interconnected port-Hamiltonian systems.                

\begin{proposition}\label{prop:mono_but_not_epi}
With the prerequisites of Definition-Proposition~\ref{def-prop:behavioural_interconnection}, we have the inclusion
\begin{align*}
\mathfrak{B}^{\mathfrak{t}}_{\mathcal{D}_1,\mathcal{R}_1,H_1}\Vert_{\mathcal{D}_{\mathrm{int}}}\mathfrak{B}^{\mathfrak{t}}_{\mathcal{D}_2,\mathcal{R}_2,H_2}\subseteq\mathfrak{B}^{\mathfrak{t}}_{\mathcal{D}_1\Vert_{\mathcal{D}_{\mathrm{int}}}\mathcal{D}_2,\mathcal{R}_1\boxplus\mathcal{R}_2,H_1+H_2},
\end{align*}
but not necessarily identity. 
\end{proposition}
\begin{proof}
The inclusion is clear from definition; it remains to find an example of port-Hamiltonian systems, where it is not the identity. In Example~\ref{ex:sufficient_not_necessary}, we constructed the Dirac structures
\begin{align*}
\mathcal{D}_{1} & = \set{(t,X,J_{\lambda}(t)X)\,\big\vert\,X\in\mathbb{R}^3,t\in\mathbb{R}},\\
\mathcal{D}_{2} & = \set{\left.\left(t,\begin{pmatrix}
x\\y
\end{pmatrix},\begin{pmatrix}
\lambda(t)^2y\\-\lambda(t)^2x
\end{pmatrix}\right)\,\right\vert\,t,x,y\in\mathbb{R}}.
\end{align*}
whose composition is the smooth Dirac structure $\mathbb{R}\times (\mathbb{R}\oplus\set{0})$. In particular, the composition
\begin{align*}
\mathcal{D}_1\vert{\mathcal{D}_2\times\mathcal{D}_2}\mathcal{D}_1 = \big(\mathcal{D}_1\circ\mathcal{D}_2\big)\times\big(\mathcal{D}_1\circ\mathcal{D}_2\big)
\end{align*}
is smooth. Let $H\in\mathcal{C}^\infty(\mathbb{R})$ be a constant function. The interconnection of the port-Hamiltonian systems
\begin{align*}
\big(\tfrac{\mathrm{d}}{\mathrm{d}t}x_j,-y_1^j,-y_2^j,\mathrm{d}H_{x_j},u_1^j,u_2^j\big)\in\mathcal{D}_1,\qquad j\in\set{1,2}
\end{align*}
with respect the interconnection constraint
\begin{align*}
\big(-y_1^j,-y_2^j,u_1^j,u_2^j\big)\in\mathcal{D}_2,\qquad j\in\set{1,2}
\end{align*}
is the port-Hamiltonian system
\begin{align*}
\big(\tfrac{\mathrm{d}}{\mathrm{d}t}x_1,\tfrac{\mathrm{d}}{\mathrm{d}t}x_2,\mathrm{d}H_{x_1},\mathrm{d}H_{x_2}\big)\in \mathbb{R}^2\oplus\set{0}.
\end{align*}
Therefore, we conclude
\begin{align*}
\mathfrak{B}^{\mathrm{c}}_{\mathcal{D}_1\Vert_{\mathcal{D}_2\times\mathcal{D}_2}\mathcal{D}_1,\set{0}\boxplus\set{0},H+H} & = \mathcal{C}^1(\cdot,\mathbb{R}^2),\\
\mathfrak{B}^{\mathrm{w}}_{\mathcal{D}_1\Vert_{\mathcal{D}_2\times\mathcal{D}_2}\mathcal{D}_1,\set{0}\boxplus\set{0},H+H} & = \big(\mathcal{C}\cap\mathcal{W}^{1,1}\big)(\cdot,\mathbb{R}^2).
\end{align*}
Let $t_0<t_1$ and $(a^1,a^2)\in\mathcal{C}^1([t_0,t_1],\mathbb{R}^2)$. $\big(-y_1^j,-y_2^j,u_1^j,u_2^j\big):[t_0,t_1]\to\mathcal{D}_2$ realises the interconnection constraint if, and only if,
\begin{align*}
u_1^j = \lambda(a)^2y_2^j & = -\lambda(a^j)\tfrac{\mathrm{d}}{\mathrm{d}t}a^j\\
-u_2^j = \lambda(a)^2y_1^j & = \lambda(a^j)^2\tfrac{\mathrm{d}}{\mathrm{d}t}a^j\\
0 & = \lambda(a^j)y_1^j+\lambda(a^j)^2y_2^j.
\end{align*}
Let $x^0\in\mathbb{R}$ be an isolated zero of $\lambda$ and consider $a^1 = a^2 = (t\mapsto x^0+t)$, where $t_0<0<t_1$ is so that $0$ is the only zero of $\lambda\circ a^j$. Therefore, we obtain
\begin{align*}
\begin{pmatrix}
y_1^j(t), y_2^j(t)
\end{pmatrix}\ \begin{cases} = \begin{pmatrix}
1, -\frac{1}{\lambda(x^0+t)}
\end{pmatrix}, & t \neq 0,\\
\in\mathbb{R}^2, & t = 0,
\end{cases}
\end{align*}
which is not continuous and may, e.g. for $\lambda(t)=(t-x_0)^2$, not be integrable.
\end{proof}

We have observed, that the behavioural interconnection is, in general, a proper subsheaf of the behaviour of the geometric interconnection. This raises the question, under which conditions the two behaviours coincide.

\section{Sufficient conditions}\label{sec:sufficient_conditions}

In this section, we shall investigate such sufficient conditions for the coincidence of the geometric and behavioural interconnection. First, we demonstrate that both interconnections coincide in the case of classical port-Hamiltonian ODEs without feedthrough.

\begin{example}\label{ex:classical_situation}
Let $U_i\subseteq\mathbb{R}^{n_i}$ be nonempty and open, let $J_1,R_i\in\mathcal{C}^\infty\big(U_i,\mathbb{R}^{n_i\times n_i}\big)$ with $J_i(\cdot) = -J_i(\cdot)^\top$ and $R_i(\cdot) = R_i(\cdot)^\top\geq 0$, and let $B_i\in\mathcal{C}^\infty\big(U_i,\mathbb{R}^{n_i\times m_i}\big)$, $i\in\set{1,2}$. Consider the port-Hamiltonian systems
\begin{equation}
\begin{aligned}
\tfrac{\mathrm{d}}{\mathrm{d}t}x_i & = \big(J_i(x_i)-R_i(x_i)\big)\nabla H_i(x_i)+B_i(x_i)u_i\\
y_i & = B_i(x_i)^\top\nabla H_i(x_i)
\end{aligned}\tag{$\mathrm{PH}_i$}
\end{equation}
with canonical geometric representation given by the Dirac structures $\mathcal{D}_i$ and non-negative Lagrangian subbundle $\mathcal{R}_i$. The canonical interconnection constraint
\begin{align}\label{eq:canonical_interconnection_constraint}
u_1 = A(x_1,x_2)y_2,\quad u_2 = -A(x_1,x_2)^\top y_1
\end{align}
with $A\in\mathcal{C}^{\infty}\big(\mathbb{R}^{n_1+n_2},\mathbb{R}^{m_1\times m_2}\big)$ is represented by the Dirac structure
\begin{align*}
\mathcal{D}_{\mathrm{int}} := \set{(x_1,x_2,y_1,y_2,A(x_1,x_2)y_2,-A(x_1,x_2)^\top y_1,u_1,u_2)\,\big\vert\,t\in\mathbb{R}, y_i\in\mathbb{R}^{m_i}}.
\end{align*}
The behavioural interconnection is then the behaviour (classical or weak) of the dynamical system
\begin{equation}\label{eq:example_interconnection_system}
\begin{aligned}
\tfrac{\mathrm{d}}{\mathrm{d}t}x_1 & = \big(J_1(x_1)-R_1(x_1)\big)\nabla H_1(x_1)+B_1(x_1)u_1\\
\tfrac{\mathrm{d}}{\mathrm{d}t}x_2 & = \big(J_2(x_2)-R_2(x_2)\big)\nabla H_2(x_2)+B_2(x_2)u_2\\
y_1 & = B_1(x_1)^\top\nabla H_1(x_1)\\
y_2 & = B_2(x_2)^\top\nabla H_2(x_2)\\
u_1 & = A(x_1,x_2)y_2,\\
u_2 & = -A(x_1,x_2)^\top y_1
\end{aligned}
\end{equation}
and the behaviour of the geometric interconnection is the behaviour of the dynamical system
\begin{small}
\begin{equation}
\begin{aligned}
\tfrac{\mathrm{d}}{\mathrm{d}t}x_1 & = \big(J_1(x_1)-R_1(x_1)\big)\nabla H_1(x_1) +  B_1(x_1)A(x_1,x_2)B_2(x_2)^\top\nabla H_2(x_2)\\
\tfrac{\mathrm{d}}{\mathrm{d}t}x_2 & = \big(J_2(x_2)-R_2(x_2)\big)\nabla H_2(x_2) - B_2(x_2)A(x_1,x_2)^\top B_1(x_1)^\top\nabla H_1(x_1)
\end{aligned}
\end{equation}
\end{small}
Let $t_0<t_1\in\mathbb{R}$, and let $(x_1,x_2)\in\mathfrak{B}_{\mathfrak{t},\mathcal{D}_1\Vert_{\mathcal{D}_{\mathrm{int}}}\mathcal{D}_2,\mathfrak{R}_1\times\mathfrak{R}_2}([t_0,t_1])$ for $\mathfrak{t}\in\set{c,w}$. In particular, $x_i\in\mathcal{C}([t_0,t_1],U_i)$ and hence smoothness of $A(\cdot,\cdot)$, $B_i(\cdot)$ and $H_i(\cdot)$ implies 
\begin{align*}
y_i(\cdot) & := B_i(x_i(\cdot))^\top\nabla H_i(x_i(\cdot))\in\mathcal{C}([t_0,t_1],\mathbb{R}^{m_i})\subseteq L^1([t_0,t_1],\mathbb{R}^{m_i}),\\
u_1(\cdot) & := A(x_1(\cdot),x_2(\cdot))y_2(\cdot)\in\mathcal{C}([t_0,t_1],\mathbb{R}^{m_1})\subseteq L^1([t_0,t_1],\mathbb{R}^{m_1}),\\
u_2(\cdot) & := -A(x_1(\cdot),x_2(\cdot))^\top y_1(\cdot)\in\mathcal{C}([t_0,t_1],\mathbb{R}^{m_2})\subseteq L^1([t_0,t_1],\mathbb{R}^{m_2}).
\end{align*}
By plugging in, it is observed that $(x_1,x_2,y_1,y_2,u_1,u_2)$ satisfies~\eqref{eq:example_interconnection_system}, and therefore $(x_1,x_2)\in\big(\mathfrak{B}^{\mathfrak{t}}_{\mathcal{D}_1,\mathcal{R}_1,H_1}\Vert_{\mathcal{D}_{\mathrm{int}}}\mathfrak{B}^{\mathfrak{t}}_{\mathcal{D}_2,\mathcal{R}_2,H_2}\big)([t_0,t_1])$. We conclude that
\begin{align*}
\mathfrak{B}^{\mathfrak{t}}_{\mathcal{D}_1,\mathcal{R}_1,H_1}\Vert_{\mathcal{D}_{\mathrm{int}}}\mathfrak{B}^{\mathfrak{t}}_{\mathcal{D}_2,\mathcal{R}_2,H_2} = \mathfrak{B}_{\mathfrak{t},\mathcal{D}_1\Vert_{\mathcal{D}_{\mathrm{int}}}\mathcal{D}_2,\mathfrak{R}_1\times\mathfrak{R}_2}.
\end{align*}
\end{example}

A second class of port-Hamiltonian systems, which is widely used, considers constant Dirac structures. In that case, we show that both interconnections coincide due to the following result.

\begin{lemma}\label{lem:constant_case}
Let $D_1\leq \mathbb{R}^{n+m}\oplus\mathbb{R}^{n+m}$ and $D_2\leq\mathbb{R}^m\oplus\mathbb{R}^m$ be constant Dirac structures. For all $t_0<t_1\in\mathbb{R}$ and $k\in\mathbb{N}\cup\set{\infty}$, $\mathcal{C}^k([t_0,t_1],\mathcal{D}_1\circ\mathcal{D}_2)$ and
\begin{align*}
\set{(X,\alpha)\in\mathcal{C}^k([a,b],\mathbb{R}^n\oplus\mathbb{R}^n)\,\left\vert
\begin{array}{l}
\exists (Y,\beta)\in\mathcal{C}^k([t_0,t_1],D_2):\\
\hspace*{1.5cm}(X,Y,\alpha,\beta)\in\mathcal{C}^k([t_0,t_1],D_1)
\end{array}
\right.}
\end{align*}
coincide, and $W^{k,p}([t_0,t_1],\mathcal{D}_1\circ\mathcal{D}_2)$ and
\begin{align*}
\set{(X,\alpha)\in W^{k,p}([a,b],\mathbb{R}^n\oplus\mathbb{R}^n)\,\left\vert
\!\begin{array}{l}
\exists (Y,\beta)\in W^{k,p}([t_0,t_1],D_2):\\
\hspace*{1cm}(X,Y,\alpha,\beta)\in W^{k,p}([t_0,t_1],D_1)
\end{array}
\right.\!}
\end{align*}
\end{lemma}
\begin{proof}
Consider a kernel representation
\begin{align*}
D_1 = \ker\begin{bmatrix}
E_1 & -E_2 & F_1 & -F_2
\end{bmatrix}
\end{align*}
with $E_1,F_1\in\mathbb{R}^{(n+m)\times n}$ and $E_2,F_2\in\in\mathbb{R}^{(n+m)\times m}$, and consider an image representation
\begin{align*}
D_2 = \set{(F_3\lambda,E_3\lambda)\,\big\vert\,\lambda\in\mathbb{R}^m}
\end{align*}
with $F_3,E_3\in\mathbb{R}^{m\times m}$. Then, we have
\begin{align*}
D_1\circ D_2 = \set{(X,\alpha)\in\mathbb{R}^n\oplus\mathbb{R}^n\,\big\vert\,\exists\lambda\in\mathbb{R}^m: E_1X+F_1\alpha = (E_2F_3+F_2E_3)\lambda}.
\end{align*}
Split $\mathbb{R}^m\simeq\mathbb{R}^\ell\oplus\ker(E_2F_3+F_2E_3)$ and decompose accordingly $F_3 = [F_3^1,F_3^2]$ and $E_3 = [E_3^1,E_3^2]$. In particular, $M := E_2F_3^1+F_2E_3^1\in\mathbb{R}^{(n+m)\times \ell}$ has full column rank and thus admits a left-inverse $N\in\mathbb{R}^{\ell\times (m+n)}$. Let $t_0<t_1\in\mathbb{R}$, $k\in\mathbb{N}$ and let $(X,\alpha)\in\mathcal{C}^k([t_0,t_1], D_1\circ D_2)$. Then,
\begin{align*}
\begin{bmatrix}
F_3^1NE_1 & F_3^1NF_1\\
E_3^1NE_1 & E_3^1NF_1
\end{bmatrix}\begin{pmatrix}
X\\\alpha
\end{pmatrix}\in\mathcal{C}^k([t_0,t_1], D_2)
\end{align*}
satisfies with $E_1X(t)+F_1\alpha(t)\in\mathrm{ran}\,M$
\begin{align*}
\begin{bmatrix}
E_2 & F_2
\end{bmatrix}\begin{bmatrix}
F_3^1NE_1 & F_3^1NF_1\\
E_3^1NE_1 & E_3^1NF_1
\end{bmatrix}\begin{pmatrix}
X(t)\\\alpha(t)
\end{pmatrix} & = MN(E_1X(t) + F_1\alpha(t))\\
& = (E_1X(t) + F_1\alpha(t))
\end{align*}
and hence $\big(X,F_3^1N(E_1X+F_1\alpha),\alpha,E_3^1N(E_1X+F_1\alpha)\big)\in \mathcal{C}^k([t_0,t_1], D_2)$. Conversely, if $(X,\alpha)\in\mathcal{C}^k([a,b],\mathbb{R}^n\oplus\mathbb{R}^n)$ with $(X,Y,\alpha,\beta)\in\mathcal{C}^k([t_0,t_1],D_1)$ for some $(Y,\beta)\in\mathcal{C}^k([t_0,t_1],D_2)$, then clearly $(X,\alpha)\in\mathcal{C}^k([t_0,t_1], D_1\circ D_2)$. This verifies the first identity. Since matrices map vector-valued Sobolev spaces $W^{k,p}$, $p\in[0,\infty]$ into vector-valued Sobolev spaces $W^{k,p}$, the second identity holds true.
\end{proof}

\begin{proposition}\label{prop:constant_case}
Let $U_j\subseteq\mathbb{R}^{n_j}$ be open and non-empty. Consider further constant Dirac structures $\mathcal{D}_j\leq U_j\times \big(\mathbb{R}^{n_j+m_j^{\mathrm{r}}+m_j^{\mathrm{ep}}+m_j^{\mathrm{ip}}}\oplus\mathbb{R}^{n_j+m_j^{\mathrm{r}}+m_j^{\mathrm{ep}}+m_j^{\mathrm{ip}}}\big)$, Lagrangian subbundles $\mathcal{R}_j\leq U_j\times\big(\mathbb{R}^{m_j^\mathrm{r}}\oplus\mathbb{R}^{m_j^\mathrm{r}}\big)$ and Hamiltonian functions $H_j\in\mathcal{C}^\infty(U_j)$. Let finally $\mathcal{D}_{\mathrm{int}}\leq (U_1\times U_2)\times\big(\mathbb{R}^{m_1^{\mathrm{ip}}+m_2^{\mathrm{ip}}}\oplus\mathbb{R}^{m_1^{\mathrm{ip}}+m_2^{\mathrm{ip}}}\big)$ be a constant Dirac structure. Then,
\begin{align*}
\mathfrak{B}^{\mathfrak{t}}_{\mathcal{D}_1,\mathcal{R}_1,H_1}\Vert_{\mathcal{D}_{\mathrm{int}}}\mathfrak{B}^{\mathfrak{t}}_{\mathcal{D}_2,\mathcal{R}_2,H_2} = \mathfrak{B}^{\mathfrak{t}}_{\mathcal{D}_1\Vert_{\mathcal{D}_{\mathrm{int}}}\mathcal{D}_2,\mathcal{R}_1\times\mathcal{R}_2,H_1+H_2},
\end{align*}
\end{proposition}
\begin{proof}
Let $(f_1,f_2,y_1,y_2,e_1,e_2,u_1,u_2)\in\mathfrak{B}_{\mathrm{c},\mathcal{D}_1\Vert_{\mathcal{D}_{\mathrm{int}}}\mathcal{D}_2,\mathcal{R}_1\times\mathcal{R}_2,H_1+H_2}\big([t_0,t_1]\big)$ for $t_0<t_1\in\mathbb{R}$. By definition, this is equivalent to
\begin{align*}
\left(-\dot{x}_1,-\dot{x}_2,f_1,f_2,y_1,y_2,\mathrm{d}H_1(x_1),\mathrm{d}H_2(x_2),e_1,e_2,u_1,u_2\right)\in\mathcal{C}\big([t_1,t_1],\mathcal{D}_1\Vert_{\mathcal{D}_{\mathrm{int}}}\mathcal{D}_2\big)
\end{align*}
and $(f_j,e_j)\in\mathcal{C}([t_0,t_1],\mathcal{R}_j)$, where $x_j := p_{U_j}y_j$ and $\dot{x}_j := \tfrac{\mathrm{d}}{\mathrm{d}t}x_j$. With
\begin{align*}
\mathcal{D}_1\Vert_{\mathcal{D}_{\mathrm{int}}}\mathcal{D}_2 = (\mathcal{D}_1\times\mathcal{D}_2)\circ\mathcal{D}_{\mathrm{int}},
\end{align*}
Lemma~\ref{lem:constant_case} implies that there exists $(\overline{y}_1,\overline{y}_2,\overline{u}_1,\overline{u}_2)\in\mathcal{C}\big([t_0,t_1],\mathcal{D}_{\mathrm{int}}\big)$ so that
\begin{align*}
\left(-\dot{x}_j,f_j,y_j,\overline{y}_j,\mathrm{d}H_j(x_j),e_j,u_j,\overline{u}_j\right)\in\mathcal{C}\big([t_0,t_1],\mathcal{D}_{j}\big),\qquad j\in\set{1,2}.
\end{align*}
Therefore, $(f_1,f_2,y_1,y_2,e_1,e_2,u_1,u_2)\in\left(\mathfrak{B}_{\mathfrak{t},\mathcal{D}_1,\mathcal{R}_1,H_1}\Vert_{\mathcal{D}_{\mathrm{int}}}\mathfrak{B}_{\mathfrak{t},\mathcal{D}_2,\mathcal{R}_2,H_2}\right)([t_0,t_1])$ and since the reverse identity holds by construction, the claimed identity is indeed true.
\end{proof}

In the proof of Lemma~\ref{lem:constant_case}, we have utilised that the clean intersection condition is, in the constant case, trivially satisfied. Indeed, we may observe that the clean intersection condition is sufficient in the non-constant case.

\begin{lemma}\label{lem:non_constant_case}
Let $\pi:E\to M$ and $F\to M$ be vector bundles over the same manifold, and let $\mathcal{D}_1\leq (E\oplus F)\oplus(E^*\oplus F^*)$ and $\mathcal{D}_2\leq F\oplus F^*$ be Dirac structures. If $\mathcal{D}_1\cap\mathcal{D}_2$ defined in~\eqref{eq:intersection} has constant rank, then $\mathcal{C}^k([t_0,t_1],\mathcal{D}_1\circ\mathcal{D}_2)$ and
\begin{align*}
\set{(X,\alpha)\in\mathcal{C}^k([a,b],E\oplus E^*)\,\left\vert
\begin{array}{l}
\exists (Y,\beta)\in\mathcal{C}^k([t_0,t_1],D_2):\\
\hspace*{1.5cm}(X,Y,\alpha,\beta)\in\mathcal{C}^k([t_0,t_1],D_1)
\end{array}
\right.}
\end{align*}
coincide for all $t_0<t_1\in\mathbb{R}$ and $k\in\mathbb{N}$; likewise, $\mathcal{L}^p_c([t_0,t_1],,\mathcal{D}_1\circ\mathcal{D}_2)$ and
\begin{align*}
\set{(X,\alpha)\in\mathcal{L}^p_c([a,b],E\oplus E^*)\,\left\vert
\begin{array}{l}
\exists (Y,\beta)\in\mathcal{L}^p_c([t_0,t_1],D_2):\\
\hspace*{1.5cm}(X,Y,\alpha,\beta)\in\mathcal{L}^p_c([t_0,t_1],D_1)
\end{array}
\right.}
\end{align*}
coincide for all $p\in[1,\infty]$.
\end{lemma}
\begin{proof}
Put $G:=p_{F\oplus F^*}\left(\mathcal{D}_1\cap\mathcal{D}_2\right)$, and let $(X,\alpha)\in \mathcal{C}^k([t_0,t_1],\mathcal{D}_1\circ\mathcal{D}_2)$ for some $t_0<t_1\in\mathbb{R}$ and $k\in\mathbb{N}$. There exists $\varepsilon\in (0,\infty)$ and representatives (or extensions) $(\overline{X},\overline{\alpha})\in\mathcal{C}^k((t_0-\varepsilon,t_1+\varepsilon),\mathcal{D}_1\circ\mathcal{D}_2)$. Put $x := \pi_E\circ \overline{X}$. Since $(t_0-\varepsilon,t_1+\varepsilon)$ is contractible, we find trivialisations
\begin{align*}
x^*(E\oplus E^*) & \simto (t_0-\varepsilon,t_1+\varepsilon)\times\mathbb{R}^{2n},\\
\varphi:x^*\mathcal{D}_1 & \simto (t_0-\varepsilon,t_1+\varepsilon)\times\mathbb{R}^{n+m},
\end{align*}
and trivialisations of $x^*G$, $x^*\mathcal{D}_2$ and $x^*(F\oplus F^*)$ so that the diagram
\begin{equation}\label{eq:comm_diagram}
\begin{tikzcd}
x^*G \arrow[d,"\subseteq"]\arrow[r,"\sim"] & (t_0-\varepsilon,t_1+\varepsilon)\times\mathbb{R}^\ell\arrow[d,"\subseteq"]\\
x^*\mathcal{D}_2\arrow[r,"\psi"]\arrow[d,"\subseteq"] & (t_0-\varepsilon,t_1+\varepsilon)\times\mathbb{R}^\ell\oplus\mathbb{R}^{m-\ell}\arrow[d,"\subseteq"]\\ 
x^*(F\oplus F^*)\arrow[r,"\sim"] & (t_0-\varepsilon,t_1+\varepsilon)\times\mathbb{R}^\ell\oplus\mathbb{R}^{m-\ell}\oplus\mathbb{R}^{m}\\ 
\end{tikzcd}
\end{equation}
commutes. In particular, we have a kernel representation
\begin{align*}
x^*\mathcal{D}_1 \simeq \set{(t,X,Y,\alpha,\beta)\in\mathbb{R}^{2(n+m)}\,\left\vert\, A(t)\begin{pmatrix}
X\\\alpha
\end{pmatrix} = B(t)\begin{pmatrix}
Y\\\beta
\end{pmatrix}\right.}
\end{align*}
for matrix fields
\begin{align*}
A & \in\mathcal{C}^\infty\big((t_0-\varepsilon,t_1+\varepsilon),\mathbb{R}^{(n+m)\times 2n}\big),\\
B = [B_1, B_2, B_3] & \in\mathcal{C}^\infty\big((t_0-\varepsilon,t_1+\varepsilon),\mathbb{R}^{(n+m)\times (\ell+(m-\ell)+m)}\big).
\end{align*}
Therefore, we have
\begin{align*}
x^*(\mathcal{D}_1\circ\mathcal{D}_2) \simeq \set{(t,X,\alpha)\in\mathbb{R}^{2n}\,\left\vert\,\begin{array}{l}\exists z_1\in\mathbb{R}^{\ell}~\exists z_2\in\mathbb{R}^{m-\ell}:\\
\hspace*{1cm} A(t)\begin{pmatrix}
X\\\alpha
\end{pmatrix} = \begin{bmatrix}B_1(t), B_2(t)\end{bmatrix}\begin{pmatrix}
z_1\\z_2
\end{pmatrix}
\end{array}\right.}
\end{align*}
and we observe that
\begin{align*}
& \ker \begin{bmatrix}B_1(t), B_2(t)\end{bmatrix}\\
&\quad \simeq \set{(Z,\beta)\in(\mathcal{D}_2)_{x(t)}\,\big\vert\,\forall (X,\alpha)\in(\mathcal{D}_1\circ\mathcal{D}_2)_{x(t)}:(X,Y,\alpha,\beta)\in\mathcal{D}_1} = G_{x(t)}
\end{align*}
for all $t\in(t_0-\varepsilon,t_1+\varepsilon)$ and thus the commutativity of~\eqref{eq:comm_diagram} yields $B_2\equiv 0$ and $\rk B_1\equiv \ell$. Hence,  $B_1$ admits the smooth left-inverse
\begin{align*}
M := (B_1(\cdot)^\top B_1(\cdot))^{-1}B_1(\cdot)^\top\in\mathcal{C}^\infty\big((t_0-\varepsilon,t_1+\varepsilon),\mathbb{R}^{\ell\times(n+m)}\big).
\end{align*}
Define
\begin{align*}
(\overline{Y},\overline{\beta}):(t_0-\varepsilon,t_1+\varepsilon) & \to F\oplus F^*,\\
t & \mapsto \psi^{-1}\left(M(t)A(t)\varphi\left(\overline{X}(t),\overline{\alpha}(t)\right)\right)
\end{align*}
which is in $\mathcal{C}^k\big((t_0-\varepsilon,t_1+\varepsilon),\mathcal{D}_2\big)$ and satisfies
\begin{align*}
\forall t\in (t_0-\varepsilon,t_1+\varepsilon): \left(\overline{X}(t),\overline{Y}(t),\overline{\alpha}(t),\overline{\beta}(t)\right)\in\mathcal{D}_1.
\end{align*}
Taking the restriction on $[t_0,t_1]$ gives the claimed identity. The second part of the lemma is proven analogously.
\end{proof}

Lemma~\ref{lem:non_constant_case} implies directly the sufficiency of the clean intersection condition for the coincidence of the geometric and behavioural interconnection.

\begin{proposition}
Let
\begin{itemize}
\item $E_s^j\to M_j$, $E_r^j\to M_j$, $E_{p_i}^j\to M_j$, $E_{p_e}^j\to M_j$ be smooth vector bundles,
\item $\mathcal{D}_j\leq \big(E_s^j\oplus E_r^j\oplus E_{p_i}^j\oplus E_{p_e}^j\big)\oplus\big((E_s^j)^*\oplus (E_r^j)^*\oplus (E_{p_i}^j)^*\oplus (E_{p_e}^j)^*\big)$ be Dirac structures,
\item $\mathcal{R}_j\leq E_r^j\oplus (E_r^j)^*$ be non-negative Lagrangian subbundles,
\item $H_j\in\mathcal{C}^\infty(M_j)$ be Hamiltonian functions,
\item $\mathcal{D}_{\mathrm{int}}\leq \big(E_{p_e}^1\oplus E_{p_e}^2\big) \big((E_{p_e}^1)^*\oplus (E_{p_e}^2)^*\big)$ be an {admissible interconnection constraint} for $(\mathcal{D}_1,\mathcal{D}_2)$.
\end{itemize}
If $(\mathcal{D}_1\times\mathcal{D}_2)\cap\mathcal{D}_{\mathrm{int}}$ has constant rank, then
\begin{align*}
\mathfrak{B}^{\mathfrak{t}}_{\mathcal{D}_1,\mathcal{R}_1,H_1}\Vert_{\mathcal{D}_{\mathrm{int}}}\mathfrak{B}^{\mathfrak{t}}_{\mathcal{D}_2,\mathcal{R}_2,H_2}\subseteq\mathfrak{B}^{\mathfrak{t}}_{\mathcal{D}_1\Vert_{\mathcal{D}_{\mathrm{int}}}\mathcal{D}_2,\mathcal{R}_1\boxplus\mathcal{R}_2,H_1+H_2}
\end{align*}
for $\mathfrak{t}\in\set{\mathrm{c},\mathrm{w}}$.
\end{proposition}
\begin{proof}
This is analogous to the proof of Proposition~\ref{prop:constant_case}.
\end{proof}

\section{Conclusion}

In this note, we have studied a behavioural approach to interconnection of port-Hamiltonian systems. We have defined classical and weak behaviours for systems defined on general manifolds. The natural properties of solutions are encoded in the language of \textbf{Int}-sheaves. The well-known interconnection utilises the composition of Dirac structures, which is well-defined provided that a clean intersection condition is satisfied. An example of port-Hamiltonian systems so that the composition of the underlying Dirac structures is again a Dirac structure, although the clean intersection condition is not satisfied, was constructed, and it was shown that the interconnected system admits a solution so that the interconnection constraint can not be regularly fulfilled, where regularity means continuity or integrability. This shows that the classical interconnection is \textit{a priori} not dynamical, but geometrical. Finally, we proved that the failure of the clean intersection condition in our example is necessary for the emergence of the gap between the geometric and behavioural interconnection. This gives a checkable condition which guarantees that the interconnection of port-Hamiltonian systems is structure preserving in both the geometric and dynamical sense, and this condition is satisfied for the practical examples known to the author.

\section*{Declaration of generative AI and AI-assisted technologies in the manuscript preparation process}

During the preparation of this work, the author used Claude in order to construct the example in the proof of Proposition~\ref{prop:mono_but_not_epi} and to find the reference~\cite{EtayGoNiSant26} in Example~\ref{ex:fun_Dirac_structures}\,(iii). After using this tool/service, the author reviewed and edited the content as needed and takes full responsibility for the content of the published article.

\section*{Acknowledgements}

The author thanks Prof.~Thomas Berger (MLU Halle-Wittenberg) for catching a mistake in Example~\ref{ex:sufficient_not_necessary}.

\bibliographystyle{plain}
\bibliography{../Literatur}

\begin{thebibliography}{10}

\bibitem{Interconnection18}
M~Barbero-Li{\~{n}}{\'{a}}n, H~Cendra, E~Garc{\'{\i}}a-Tora{\~{n}}o
  Andr{\'{e}}s, and D~Mart{\'{\i}}n de~Diego.
\newblock New insights in the geometry and interconnection of
  port-{H}amiltonian systems.
\newblock {\em Journal of Physics A: Mathematical and Theoretical},
  51(37):375201, aug 2018.

\bibitem{Bock26}
Alexander~Samuel Bock.
\newblock On local solutions to time-varying linear daes, 2026.

\bibitem{Bred97}
Glen~E. Bredon.
\newblock {\em Sheaf Theory}, volume 170 of {\em Graduate Texts in
  Mathematics}.
\newblock Springer-Verlag, Berlin, Heidelberg, New York, 2nd edition, 1997.

\bibitem{Burs13}
Henrique Bursztyn.
\newblock A brief introduction to {D}irac manifolds.
\newblock In Alexander Cardona, Iván Contreras, and Andrés~F. Reyes-Lega,
  editors, {\em Geometric and topological methods for quantum field theory},
  pages 4--38. Cambridge University Press, 2013.

\bibitem{Interconnection07}
J.\ Cervera, A.J.\ van~der Schaft, and A.\ Banos.
\newblock Interconnection of port-{H}amiltonian systems and composition of
  {D}irac structures.
\newblock {\em automatica}, 43:212 -- 225, 2007.

\bibitem{ConvSchaf19}
Alexandra Convent and Jean~Van Schaftingen.
\newblock Higher order intrinsic weak differentiability and {S}obolev spaces
  between manifolds.
\newblock {\em Advances in Calculus of Variations}, 12(3):303--332, 2019.

\bibitem{ConvScha16}
Alexandra Convent and Jean van Schaftingen.
\newblock Intrinsic colocal weak derivatives and sobolev spaces between
  manifolds.
\newblock {\em Annali della Scuola Normale Superiore di Pisa - Classe di
  Scienze}, XVI:97--128, 2016.

\bibitem{CourWein88}
Ted Courant and Alan Weinstein.
\newblock Beyond {P}oisson structures.
\newblock In {\em Actions hamiltoniennes de groupes. Troisième théorème de
  Lie (Lyon, 1986)}, pages 39--49, Paris, 1988. Hermann.

\bibitem{Cour90}
Theodore~James Courant.
\newblock Dirac manifolds.
\newblock {\em Transactions of the American Mathematical Society},
  319(2):631--661, 1990.

\bibitem{DalsvdS98}
Morten Dalsmo and Arjan {van der Schaft}.
\newblock On representations and integrability of mathematical structures in
  energy-conserving physical systems.
\newblock {\em SIAM Journal of Control and Optimization}, 37(1):54--91, 1998.

\bibitem{Dorf93}
Irene Dorfman.
\newblock {\em Dirac {S}tructures and {I}ntegrability of {N}onlinear
  {E}volution {E}quations}.
\newblock John Wiley and Sons, New York, Chichester, Brisbane, Toronto,
  Singapore, 1993.

\bibitem{Dorf87}
Irene~Ya. Dorfman.
\newblock Dirac structures of integrable evolution equations.
\newblock {\em Physics Letters A}, 125(5):240--246, 1987.

\bibitem{EtayGoNiSant26}
Fernando Etayo, Pablo Gómez-Nicolás, and Rafael Santamaría.
\newblock On the triviality of the generalized tangent bundle, 2026.

\bibitem{Hajl09}
Piotr Hajlasz.
\newblock {\em Sobolev Mappings between Manifolds and Metric Spaces}, pages
  185--222.
\newblock Springer New York, New York, NY, 2009.

\bibitem{SchaftJeltsema14}
Dimitri Jeltsema and Arjan van~der Schaft.
\newblock Port-{H}amiltonian systems theory: An introductory overview.
\newblock In {\em Foundations and Trends in Systems and Control}, volume~1,
  pages 173--378. now Publishers Inc., 2014.

\bibitem{John02}
Peter~T. Johnstone.
\newblock {\em Sketches of an Elephant. A Topos Theory Compendium. Volume 1}.
\newblock Clarendon Press, Oxford, 2002.

\bibitem{Lee13}
John~M.\ Lee.
\newblock {\em Introduction to {S}mooth {M}anifolds}.
\newblock Springer Verlag, New York, 2nd edition, 2013.

\bibitem{LibeMarl87}
Paulette Libermann and Charles-Michel Marle.
\newblock {\em Symplectic Geometry and Analytical Mechanics}.
\newblock Mathematics and Its Applications. D. Reidel Publishing Company,
  Dordrecht, Boston, Lancaster, Tokyo, 1987.

\bibitem{MaLaMoer92}
Saunders {Mac Lane} and Ieke Moerdijk.
\newblock {\em Sheaves in Geometry and Logic. A First Introduction to Topos
  Theory}.
\newblock Springer-Verlag, New York, 1992.

\bibitem{Merk09}
Jochen Merker.
\newblock On the {G}eometric {S}tructure of {H}amiltonian {S}ystems with
  {P}orts.
\newblock {\em Journal of Nonlinear Science}, 19:717--738, 2009.

\bibitem{Pala68}
Richard~S. Palais.
\newblock {\em Foundations of global non-linear analysis}.
\newblock W.~A.~Benjamin, inc., New York, Amsterdam, 1968.

\bibitem{Polderman}
Jan~Willem Polderman and Jan~C. Willems.
\newblock {\em Introduction to Mathematical Systems Theory}.
\newblock Springer Verlag, London, Heidelberg, New York, 1998.

\bibitem{Reis25b}
Timo Reis.
\newblock Weak solutions of port-{H}amiltonian systems, 2025.

\bibitem{SchuSPivVasi20}
Patrick Schultz, David~I. Spivak, and Christina Vasilakopoulou.
\newblock Dynamical {S}ystems and {S}heaves.
\newblock {\em Applied Categorical Structures}, 28:1--57, 2020.

\bibitem{Vyso20}
Jan Vysoký.
\newblock Hitchhiker's guide to {C}ourant algebroid relations.
\newblock {\em Journal of Geometry and Physics}, 151:103635, 2020.

\end{thebibliography}
\end{document}